\documentclass{amsart}
\usepackage{amsmath,amssymb,bbm}
\usepackage{amsfonts,amssymb}
\usepackage{xcolor}
\usepackage{enumerate}

\usepackage[numeric]{amsrefs}
\UseRawInputEncoding

\theoremstyle{plain} 

\newtheorem{thm}{Theorem}[section]
\newtheorem{cor}[thm]{Corollary}
\newtheorem{lem}[thm]{Lemma}
\newtheorem{prop}[thm]{Proposition}

\newtheorem{defn}[thm]{Definition}

\theoremstyle{remark}
\newtheorem{rem}[thm]{Remark}
\def\Ld{\Lambda}

\numberwithin{equation}{section}

\def\bfi {\mathbf{i}}
\def\bfj{\mathbf{j}}

\def\f{\frac}

\def\vi{\varphi}
\def\({\left(}
\def \){ \right)}

  \def\sph{\mathbb{S}^{d-1}}

\def\Bl{\Bigl}
\def\Br{\Bigr}
 
 \def\ee{{\textnormal{e}}}
 
 \def\tr{{\triangle}}
 \def\Ga{\Gamma}

\def\ta{\theta}
\def\al{{\alpha}}
\def\da{{\delta}}
\def\sa{{\sigma}}
 
 \def\b{{\beta}}
 
  \def\ga{{\gamma}}
 
 \def\t{{\theta}}

 \def\va{\varepsilon}

 \def\ib{{\mathbf i}}
 
 \def\kb{{\mathbf k}}

 \def\NN{{\mathbb N}}

 \def\RR{{\mathbb R}}
 \def\SS{{\mathbb S}}

 \def\ZZ{{\mathbb Z}}

  \def\proj{\operatorname{proj}}

  \def\sph{\mathbb{S}^{d-1}}
  
\def\Og{\Omega}

\def\al{\alpha}
\newcommand{\wt}{\widetilde}
\newcommand{\wh}{\widehat}
\def\sub{\substack}

\def\p{\partial}
\def\ld{\lambda}
\def\bl{\bigl}
\def\br{\bigr}
\def\og{\omega}
\def\Ld{\Lambda}

 \def\CE{\mathcal{E}}
 \def\EEE{{\mathcal E}}

\newcommand{\ds}{\displaystyle}
\newcommand{\R}{{\mathbb{R}}}

\newcommand{\diam}{{\rm diam}}
\newcommand{\dist}{{\rm dist}}

\def\be{\begin{equation}}
\def\ee{\end{equation}}

\begin{document}

\title[]{Polynomial approximation on $C^2$-domains}
\author{Feng Dai}
\address{F.~Dai, Department of Mathematical and Statistical Sciences\\
	University of Alberta\\ Edmonton, Alberta T6G 2G1, Canada.}
\email{fdai@ualberta.ca}

\author{Andriy Prymak}
\address{Department of Mathematics, University of Manitoba, Winnipeg, MB, R3T2N2, Canada}

\email{prymak@gmail.com}

\thanks{	The first author was supported by  NSERC of Canada Discovery
	grant RGPIN-2020-03909, and the second author  was supported by NSERC of Canada Discovery grant RGPIN-2020-05357.
	}



\begin{abstract}
	 We  introduce appropriate computable moduli of smoothness
to 
characterize
the rate of best  approximation  by multivariate
polynomials on a connected and  compact $C^2$-domain $\Omega\subset \mathbb{R}^d$. This  new modulus of smoothness is defined via finite differences along the directions of coordinate axes, and along a number of  tangential directions from the boundary.  With this modulus, we  prove  both the  direct Jackson inequality    and the corresponding  inverse for the best polynomial approximation in  $L_p(\Omega)$. The Jackson inequality is established for the full range of $0<p\leq \infty$,  while  its proof  relies on a recently established Whitney type  estimates with constants depending only on certain  parameters; and on a highly localized  polynomial partitions of  unity on a $C^2$-domain which is of independent interest. The inverse inequality is established for $1\leq p\leq \infty$, and its  proof relies on a recently proved Bernstein type inequality associated with the tangential derivatives on the boundary of $\Omega$.  Such an  inequality also allows us to establish the inverse theorem for Ivanov's  average moduli of smoothness on general compact $C^2$-domains.

\end{abstract}

\maketitle

\section{Introduction and Main Results}

\subsection{Historical remarks}

One of the primary questions of approximation theory is to characterize the
rate of approximation by a given system in terms of some  modulus
of  smoothness.   It is  well known (see, e.g.~\cite{De-Lo,Di-To,Ni}) that  the quality of approximation  by algebraic polynomials  increases towards the boundary of the underlying domain. As a result,  characterization of the class of functions with a prescribed rate of best  approximation by algebraic polynomials on a compact domain with nonempty boundary 
cannot be described by the ordinary moduli of smoothness. 
Several  successful  moduli of smoothness were introduced to solve  this problem in the setting of one variable. Among them the most established   ones are the  Ditzian-Totik moduli of smoothness  \cite{Di-To} and the average moduli of smoothness of K. Ivanov  \cite{Iv2}  (see the survey paper  \cite{Dit07} for details).  
The essential idea is that for the same approximation rate one may allow the function to be much less smooth closer to the endpoints of the interval. 
Successful  attempts were also   made to solve the problem in more variables,    the  most notable  being the work  of K. Ivanov   for polynomial approximation on piecewise $C^2$-domains in $\RR^2$  \cite{Iv}, and the recent works of Totik for polynomial approximation  on general  polytopes and algebraic domains  \cite{To14, To17}; we will describe~\cite{Iv} and~\cite{To17} in more details below. The following list is not meant to be exhaustive, but we would like to also mention several other related works: results for simple polytopes by Ditzian and Totik~\cite{Di-To}*{Chapter~12}, an announcement of a characterization of approximation classes by Netrusov~\cite{Ne}, possibly reduction to local approximation by Dubiner~\cite{Du}, results for simple polytopes for $p<1$ by Ditzian~\cite{Di96}, a new modulus of smoothness and characterization of approximation classes on the unit ball by the first author and Xu~\cite{DX}, a different alternative approach on the unit ball by Ditzian~\cite{Di14a,Di14b}, and a strengthening of the rate of polynomial approximation near conic boundary points of general convex domains by Yu.~Brudnyi~\cite{Br}.

The  main aim in this paper is to introduce a  computable  modulus of smoothness for functions on  $C^2$-domains,  for which both the direct Jackson inequality and  the corresponding  converse hold. As is well known,  the definition of  such a modulus must take into account  the boundary of the underlying domain. 

We start with some necessary notations. 
Let $L^p(\Og)$, $0<p<\infty$ denote the Lebesgue $L^p$-space defined with respect to the Lebesgue measure on a compact domain $\Og\subset \RR^d$. In the limit case we set $L^\infty(\Og)=C(\Og)$, the space of all continuous functions  on $\Og$ with the uniform norm $\|\cdot \|_\infty$.
Given  $\xi, \eta\in \RR^d$, and  $r\in\NN$,  we define 
$$ \tr_\xi^r f(\eta) :=\sum_{j=0}^r (-1)^{r+j} \binom{r} {j} f(\eta+j\xi ),$$
where we assume that $f$ is defined everywhere on the set  $\{\eta+j\xi:\  \   j=0,1,\dots, r\}$.
For a function $f: \Og\to\RR$, we also define  
\begin{equation}\label{finite-diff}
	\tr_\xi^r (f, \Og, \eta):=\begin{cases}
		\tr_\xi^r f(\eta),\   \ &\text{if  $[\eta, \eta+r\xi]\subset \Og$,}\\
		0, &\   \  \text{otherwise},
	\end{cases}
\end{equation}
where $[x,y]$ denotes the line segment connecting any two points $x,y\in\RR^d$. The symmetric versions of these finite differences are
\[
\wt \tr^r_\xi f(\eta):= \tr^r_\xi f\Bl(\eta-\f r2 \xi\Br)
\quad\text{and}\quad
\wt \tr^r_\xi (f,\Og,\eta):= \tr^r_\xi \Bl(f,\Og,\eta-\f r2 \xi\Br).
\]
The best approximation of  $f\in L^p(\Og)$ by means of algebraic polynomials of total degree at most $n$ is defined as
$$ E_n(f)_p=E_n(f)_{L^p(\Og)}:=\inf\Bl\{ \|f-Q\|_p:\   \  Q\in \Pi^d_n\Br\},$$
where $\Pi^d_n$ is the space of algebraic polynomials of total degree $\le n$ on $\RR^d$.
Given a set  $E\subset \RR^d$, we denote by $|E|$ its Lebesgue measure in $\RR^d$, and   define  $\dist(\xi, E):=\inf_{\eta\in E}\|\xi-\eta\|$ for $\xi\in \RR^d$, (if $E=\emptyset$, then define $\dist(\xi, E)=1$). Here and throughout the paper,   $\|\cdot\|$ denotes the Euclidean norm. 
Finally, let $\sph\subset \RR^d$ be the unit sphere of $\RR^d$, and let $e_1=(1,0,\dots, 0), \dots, e_d =(0, \dots, 0, 1)$ denote the standard canonical basis  in $\RR^d$.

Next, we describe   the work of K. Ivanov \cite{Iv}, where  a new modulus of smoothness was introduced  to study the best algebraic polynomial approximation
for functions of two variables on a bounded domain   with piecewise $C^2$ boundary.  To avoid technicalities,  we always   assume  that  $\Og\subset \RR^d$ is  the closure of an open,  bounded, and connected  domain in $\R^d$    with $C^2$ boundary $\Ga$ (see Definition~\ref{def:c2}). Consider the following metric on $\Og$:
\begin{equation}\label{metric}\rho_\Og (\xi,\eta):=\|\xi-\eta\|+ \Bl|\sqrt{\dist(\xi, \Gamma)} -\sqrt{\dist(\eta, \Gamma)}\Br|,\   \  \xi, \eta\in\Og. \end{equation}
For $\xi\in\Og$ and $t>0$, set 
$ U( \xi, t):= \{\eta\in\Og:\  \  \rho_\Og(\xi,\eta) \leq t\}$.
For $0< q\leq p\leq \infty$,  the average   $(p,q)$-modulus of order $r\in\NN$ of $f\in L^p(\Og)$  was defined in  \cite{Iv} by\footnote{Both the metric $\rho_\Og$ and  the average moduli of smoothness $\tau_r(f,t)_{p,q}$  were defined in \cite{Iv} for a more general domain $\Og\subset \RR^2$. } 
\begin{equation}\label{eqn:ivanov} \tau_r (f; \da)_{p,q} :=\Bl\| w_r (f, \cdot, \da)_q \Br\|_p,\end{equation}
where    
$$ w_r (f, \xi, \da)_q : =\begin{cases}
	\displaystyle \Bl( \f 1 {|U(\xi,\da)|} \int_{U(\xi,\da)} |\tr_{(\eta-\xi)/r} ^r (f,\Og,\xi)|^q \, d\eta\Br)^{\f1q},\  \  & \text{if $0<q <\infty$};\\
	\sup_{\eta\in U( \xi,\da)} |\tr_{(\eta-\xi)/r}^r (f,\Og,\xi)|,\   \ &\text {if $q=\infty$}.\end{cases}$$
Intuitively, the smoothness is measured through local subdomains $U(\xi,t)$. When $\xi\in\Gamma$, $U(\xi,n^{-1})$ has the (Euclidean) size roughly $n^{-1}$ in the directions parallel to $\Gamma$ at $\xi$ (tangential directions), while in the (orthogonal) direction of the inward normal the size will be roughly $n^{-2}$. Thus, for the same approximation rate, the function is allowed to be less smooth in the inward normal direction. This is natural to expect as we do not worry about the values of the approximating polynomial outside of the domain. On the other hand, one also needs to account for the varying (in arbitrary $C^2$ manner) throughout the domain tangential directions, which is one of the key difficulties.

With the modulus defined in~\eqref{eqn:ivanov},  the following result was  announced  without proof in  \cite{Iv} for a   bounded domain  in the plane   with piecewise $C^2$ boundary.
\begin{thm}\label{thm:ivanov}\cite{Iv} Let $\Og$ be the closure of a bounded open domain in the plane $\RR^2$ with piecewise $C^2$-boundary $\Ga$. 	
	If $f\in L^p(\Og)$, $1\leq q\leq p \leq \infty$ and $r\in\NN$, then
	\begin{equation}\label{1-4-0} E_n (f)_p \leq C_{r, \Og} \tau_r (f, n^{-1})_{p,q}.\end{equation}
	Conversely, if either $p=\infty$ or $\Og$ is a parallelogram or a disk and $1\leq p\leq \infty$, then
	\begin{equation}\label{1-5-0}\tau_r (f, n^{-1})_{p,q} \leq C_{r,\Og} n^{-r} \sum_{s=0}^n (s+1)^{r-1} E_s (f)_p.\end{equation}
\end{thm}

It remained open in \cite{Iv} whether the inverse inequality~\eqref{1-5-0} holds for the full range of $1\leq p\leq \infty$  for more general  $C^2$-domains other than parallelograms and  disks.  The methods developed in this paper allow us to give a positive answer to this question.  In fact, we shall prove the Jackson inequality~\eqref{1-4-0} for $0<p\leq \infty$ and the inverse inequality~\eqref{1-5-0} for $1\leq p\leq \infty$ for all compact, connected $C^2$-domains $\Og\subset \RR^d$. Our results   apply to higher dimensional domains as well.

Finally, we describe the recent  work of Totik \cite{To17}, where a new modulus of smoothness using the univariate moduli of smoothness on circles and line segments was introduced to study polynomial approximation on  algebraic domains.
Let  $\Og\subset \RR^d$ be the closure of a bounded, finitely connected domain with $C^2$ boundary $\Ga$. Such a domain  is called an algebraic domain  if for   each connected  component $\Ga'$ of the boundary $\Ga$,   there  is  a  polynomial $\Phi(x_1,\dots, x_d)$ of $d$ variables  such that $\Ga'$  is one of the components of  the surface $\Phi(x_1,\dots, x_d)=0$ and $\nabla \Phi(\xi) \neq 0$ for each $\xi\in\Ga'$.   
The $r$-th order  modulus of smoothness of $f\in C(\Og)$ on a circle $\mathcal{C}\subset \Og$ is defined  as in the classical trigonometric approximation theory by 
\begin{align*}
	\wh{\og}_{\mathcal{C}} ^r (f,t) :
	&=\sup_{0\leq \ta \leq t} \sup_{0\leq \vi\leq 2\pi } 
	\Bl| \wt\tr^r_\theta f_{\mathcal{C}}(\vi) \Br|
\end{align*}
where  we identify the circle $\mathcal{C}$ with the interval $[0, 2\pi)$ and  $f_{\mathcal{C}}$ denotes   the restriction of $f$ on $\mathcal{C}$.
Similarly, if $I=[a,b]\subset \Og$ is a line segment and $e\in\SS^{d-1}$ is the direction of $I$, then with  $\wt{d}_I (e, z): =\sqrt{\|z-a\|\|z-b\|}$,  we may define 
the modulus of smoothness of $f\in C(\Og)$   on $I$  as  
\begin{align*}
	\wh{\og}_I^r (f,\da)
	&=\sup_{0\leq h\leq \da} \sup_{z\in I}
	\Bl| \wt\tr^r_{h \wt{d}_I (e, z)e}(f,\Omega,z) \Br|. 
\end{align*}
Now we  define the $r$-th order  the modulus of smoothness of $f\in C(\Og)$ on the domain $\Og$ as 
\begin{equation}\label{eqn:mod-totik}
	\wh{\og}^r(f,\da)_\Og=\max\Bl( \sup_{\mathcal{C}_\rho} \wh{\og}_{\mathcal{C}_\rho}^r (f,\da), \sup_I \wh{\og}_I^r (f,\da)\Br),
\end{equation}
where the suprema are taken for all circles $\mathcal{C}_\rho\subset \Og$ of some radius $\rho$  which are parallel with a coordinate plane,  and for all segments $I\subset \Og$ that are parallel with one of the coordinate axes. 
With this modulus of smoothness,  Totik  proved  

\begin{thm}\label{Totik:thm}\cite{To17}
	If $\Og\subset \RR^d$ is an algebraic domain and $f\in C(\Og)$, then 
	\begin{equation} \label{1-6-0}E_n (f)_{C(\Og)} \leq C \wh{\og}^r (f, n^{-1})_{\Og},\   \ n\ge rd,\end{equation}
	and 
	\begin{equation} \label{1-7-0} \wh{\og}^r(f, n^{-1})_{\Og}  \leq C n^{-r} \sum_{k=0}^n (k+1)^{r-1} E_k (f)_{C(\Og)} \end{equation}
	with a constant $C$ independent of $f$ and $n$.
\end{thm}

From the classical inverse inequalities in one variable, and 
the way the moduli of smoothness  $\wh{\og}^r(f,t)_\Og$ are defined,  one can easily show that the inverse inequality~\eqref {1-7-0} in fact  holds on more general $C^2$-domains $\Og$. On the other hand,  however,    it is much harder to show   the direct Jackson inequality~\eqref{1-6-0} even on algebraic domains (see \cite{To17}).\footnote{{In a private  communication, V. Totik kindly showed us that certain  quasi-Whitney inequality  can be established for the  moduli $\wh{\og}^r$ on cells of distance $≥ C/n^2$ from
		the boundary of $\Og$, which, combined with certain techniques from Section \ref{ch:direct} of the current paper, will yield  the Jackson inequality \eqref{1-6-0} for the  moduli $\wh{\og}^r$  on a general $C^2$-domain.  }}
Furthermore, it  is unclear   how to extend the results of Theorem~\ref{Totik:thm} to  $L^p$ spaces with $p<\infty$.

In this paper, we will introduce a new computable modulus of smoothness on a connected, compact $C^2$-domain $\Og\subset \RR^d$. Our new modulus of smoothness is defined via finite differences along the directions of coordinate axes, and along 
tangential directions on the boundary.  With this modulus, we shall prove  a  direct Jackson-type inequality   for the full range of $0<p\leq \infty$, and the corresponding  inverse for $1\leq p\leq \infty$.  The proof of the Jackson inequality relies on a Whitney type estimate on certain domains of special type which we recently established in~\cite{Da-Pr-Whitney}, and a polynomial partition of unity on $\Og$ which we construct motivated by the ideas of Dzjadyk and Konovalov~\cite{Dz-Ko}. On the other hand, the proof of the inverse inequality is more difficult. It 
relies on a new tangential Bernstein inequality on $C^2$-domains, which we recently established in~\cite{Da-Pr-Bernstein}. 

We  give some preliminary materials in the next subsection. 
After that, we define the new modulus of smoothness in Section~\ref{modulus:def}. The main results of this paper are summarized in  Section~\ref{summary}, where we also  describe briefly the organization of the rest of the paper.  

\subsection{Preliminaries }\label{preliminary}

We start with a brief description of some necessary notations. Often we will work with domains bounded by graphs of functions, so it will be more convenient  to  work on the $(d+1)$-dimensional Euclidean space $\RR^{d+1}$ rather than    the  $d$-dimensional  space $\RR^d$. We shall often    write a point in $\RR^{d+1}$ in the form $(x, y)$  with  $x=(x_1,\dots, x_d)\in\RR^d$ and $y=x_{d+1}~\in~\RR$.     
Let $B_r[\xi]$ (resp., $B_r(\xi)$ )  denote  the closed ball (resp., open ball)  in $\R^{d+1}$ centered at $\xi\in\RR^{d+1}$   having radius $r>0$.  
A   rectangular box in $\RR^{d+1}$ is a set that  takes the form $[a_1,b_1]\times \dots\times [a_{d+1}, b_{d+1}]$ with $-\infty<a_j<b_j<\infty$,  $j=1,\dots, d+1$.  
We  always assume that the sides of a  rectangular box   are parallel with  the coordinate axes.   If $R$ denotes either a parallelepiped or a ball in $\RR^{d+1}$, then  we  denote by $cR$ the dilation of  $R$ from its center by a factor $c>0$.
Given  $1\leq i\neq j\leq d+1$,  we call   the coordinate  plane spanned by the vectors $e_i$ and $e_j$  the $x_ix_j$-plane.  Finally, we use the notation  $A_1\sim A_2$ to mean that there exists a positive constant $c>0$ such that $c^{-1}A_1\leq A_2\leq c A_1$.

\subsection{Directional moduli of smoothness} 

The $r$-th order   directional modulus of smoothness  on a domain   $\Og\subset \RR^{d+1}$  along  a  set $\mathcal{E}\subset \SS^d$ of directions    is defined  by 
$$ \og^r (f,  t; \mathcal{E})_p:=\sup_{\xi\in \mathcal{E}} \sup_{0<u\leq t } \|\tr_{u\xi}^r (f, \Og, \cdot)\|_{L^p(\Og)}=
\sup_{\xi\in \mathcal{E}} \sup_{0<u\leq t } \|\tr_{u\xi}^r f\|_{L^p(\Og_{ru\xi})},$$
where $\tr_{u\xi}^r f=\tr_{u\xi}^r(f,\Og, \cdot)$ is  given in~\eqref{finite-diff}, and 
$\Og_{\eta}:=\{\xi\in \Og:\  \  [\xi, \xi+\eta]\subset \Og\}$ for $\eta\in\RR^{d+1}$. 
Let 
$$ \og^r(f, \Og; \mathcal{E})_p:=
\og^r (f,  \diam(\Og); \mathcal{E})_p,$$ where $\diam (\Og) :=\sup_{\xi, \eta \in \Og} \|\xi-\eta\|$. 
If $\mathcal{E}=\SS^d$, then we write $\og^r (f,  t)_p=\og^r (f,  t; \SS^d)_p$ and  $\og^r(f, \Og)_p= \og^r(f, \Og; \SS^d)_p$, whereas if $\mathcal{E}=\{e\}$ contains only one direction $e\in\SS^d$,  we write  $\og^r (f,  t; e)_p=\og^r (f,  t; \mathcal{E})_p$ and  $\og^r(f, \Og; e)_p= \og^r(f, \Og; \mathcal{E})_p$.
We shall frequently use the following two properties of these directional moduli of smoothness, which can be easily verified from the definition: \begin{enumerate}[\rm (a)] 
	\item For each $\mathcal{E}\subset \SS^d$,   \begin{equation*}\label{2-1-18}\og^r (f,  \Og; \mathcal{E})_p=\og^r (f,  \Og; \mathcal{E}\cup(-\mathcal{E}))_p.\end{equation*}
	\item   If $T$ is an affine mapping given by  $T\eta =\eta_0 +T_0\eta$ for all $\eta\in\R^{d+1}$ with   $\eta_0\in\R^{d+1}$ and   $T_0$ being  a nonsingular linear mapping on $\R^{d+1}$,   then 
	\begin{equation*}\label{5-1-eq}
		\og^r (f,   \Og; \mathcal{E})_p =\Bl|\text{det} \ (T_0)\Br|^{-\f1p}  \og^r( f\circ T^{-1},  T(\Og); \mathcal{E}_{T})_p,
	\end{equation*}
	where  
	$\mathcal{E}_{T}=\bl\{\f{ T_0 x}{\|T_0 x\|}:\  \ x\in \mathcal{E}\br\}$.   Moreover, if   $\xi, e\in\SS^d$ is such that $e=T_0(\xi)$, then for any $h>0$,
	\begin{equation}\label{2-3-18}
		|\text{det} \  (T_0)|^{\f1p} 	\|\tr_{he} ^r ( f, \Og)\|_{L^p(\Og)}^p=\|\tr_{h\xi}^r (f\circ T^{-1}, T(\Og))\|_{L^p(T(\Og))}.
	\end{equation}	 
	
\end{enumerate} 
Next, we  recall that  the  analogue of the  Ditzian-Totik modulus on   $\Og\subset \RR^{d+1}$  along a direction $e\in\SS^d$    is defined as  (see \cite{To14, To17}):
\begin{equation}\label{2-3-DT}\og_{\Og, \vi}^r(f,t; e)_{p}:=
	\sup_{|h|\leq \min\{t,1\}} \Bl\|\wt\tr_{h\vi_\Og (e, \cdot) e}^r (f, \Og, \cdot)\Br\|_{L^p(\Og)},\      \   t>0,\end{equation}
where 
\begin{equation}\label{funct-vi} \vi_\Og (e, \xi):=\max\Bl\{ \sqrt{l_1l_2}:\  \  l_1, l_2\ge 0,\  \ [\xi-l_1 e, \xi+l_2 e]\subset \Og\Br\},\  \ \xi\in\Og.\end{equation}
For simplicity, we also define  $\vi_\Og (\da e, \xi)=\vi_\Og (e, \xi)$ for $e\in\SS^d$,  $\da>0$ and $\xi\in\Og$.

\subsection{Domains of special type}




A set  $G\subset \RR^{d+1}$ is called  an {\sl upward}  $x_{d+1}$- domain with base size $b>0$ and parameter $L\ge 1$  if 
it can be written in the form 
\begin{equation} \label{2-7-special}G=\xi+\Bl\{( x,  y):\  \  x\in (-b,b)^d,\   \  g(x)- L b< y \leq g(x)\Br\} \end{equation}
with   $\xi\in\RR^{d+1}$ and   $g\in C^2(\RR^d)$. 
For such a domain $G$,   and a parameter  $\ld\in (0, 2]$,  we define 
\begin{align*}
	G(\ld):&=\xi +\Bl\{ (x,   y):\  \  x\in (-\ld b, \ld b)^d,\   \   g(x)-\ld L b < y \leq g(x)\Br\},\\
	\p'G(\ld)&:=\xi +\Bl\{ (x,  g(x)):\  \  x\in (-\ld b, \ld b)^d\Br\}.
\end{align*} 
Associated with the set $G$ in~\eqref{2-7-special}, we also define 
\begin{align*}
	G^\ast:&=\xi+\Bl\{( x,   y):\  \  x\in (-2b,2b)^d,\   \ \min_{u\in [-2b, 2b]^d} g(u)-4Lb <  y \leq g(x)\Br\}.\label{G}
\end{align*}


For later applications, we give the following remark on the above definition.
\begin{rem}\label{rem-2-1-0}
	In the above definition,
	we may choose the base size  $b$    as small as we wish, and we   may also  assume the parameter $L$ in~\eqref{2-7-special} satisfies 
	\begin{equation}\label{parameter-2-9}
		L\ge L_G:=4\sqrt{d} \max_{x\in [-2b, 2b]^d} \|\nabla g(x)\| +1,
	\end{equation}
	since otherwise we may consider a   subset of the form 
	\begin{equation*} G_0=\xi+\Bl\{( x,  y):\  \  x\in (-b_0,b_0)^d,\   \  g(x)- L_0 b_0< y \leq g(x)\Br\}\end{equation*}
	with  $L_0=Lb/b_0$ and $b_0\in (0, b)$ being  a sufficiently small constant. 
	Unless otherwise stated, we will always assume that  the condition~\eqref{parameter-2-9} is satisfied for each upward $x_{d+1}$-domain.

	%
	%
	%
\end{rem}

We  may  define 
an upward $x_j$-domain $G\subset \RR^{d+1}$    and the associated   sets $G(\ld)$,   $\p' G(\ld)$, $G^\ast$  for  $1\leq j\leq d$  in a similar manner, using  the reflection  
$$\sa_j (x) =(x_1,\dots, x_{j-1}, x_{d+1}, x_{j+1},\dots, x_d, x_j),\  \  x\in\RR^{d+1}.$$
Indeed,  $G\subset \RR^{d+1}$
is  an  {\it upward} $x_j$-domain with base size $b>0$  and parameter  $L\ge 1$ if  $E:=\sa_j (G)$ is an  upward $x_{d+1}$-domain with base size $b$ and parameter $L$, in which case we define   
$$G (\ld) = \sa_j \bl( E (\ld)\br),\    \ 
\p' G(\ld)= \sa_j\bl(\p' E (\ld)\br),\   \   \   G^\ast =\sa_j (E^\ast).$$

We can also  define a  {\it downward} $x_j$-domain   and the associated   sets $G(\ld)$,    $\p' G(\ld)$, using  the reflection with respect to the coordinate plane $x_j=0$:  
$$\tau_j (x) :=(x_1, \dots, x_{j-1}, -x_j, x_{j+1}, \dots, x_{d+1}),\   \ x\in\RR^{d+1}.$$
Indeed,   $G\subset \RR^{d+1}$
is  an  {\it downward} $x_j$-domain with base size $b>0$ and parameter $L\ge 1$  if  $H:=\tau_j (G)$ is an  upward $x_{j}$-domain with base size $b$ and parameter $L\ge 1$, in which case we define 
$$G (\ld) = \tau_j \bl( H (\ld)\br),\   \
\p' G(\ld)= \tau_j\bl(\p' H (\ld)\br),\   \    G^\ast =\tau_j(H^\ast).$$

%
%
%

We  say     $G\subset  \RR^{d+1}$  is  a domain  of special type   if it  is an upward or downward $x_j$-domain for some $1\leq j\leq d+1$, in which case   we call $\p' G(\ld)$ the essential boundary of $G(\ld)$, and write      $\p' G =\p' G(1)$ and $\p' G^\ast =\p' G (2)$.

\begin{defn}\label{Def-2-1}   Let $\Og\subset \RR^{d+1}$ be a bounded domain with boundary $\Ga=\p \Og$, and let $G\subset \Og$ be a domain of special type. 
	We say $G$ is attached to $\Ga$ if $\overline{G}^\ast\cap \Ga =\overline{\p' G^\ast}$ 
	and there exists an open rectangular box $Q$ in $\RR^{d+1}$  such that $ G^\ast =Q\cap \Og$.
\end{defn}


\subsection*{$C^2$-domains}

In this paper, we shall mainly  work on   $C^2$-domains, defined as follows: 

\begin{defn}\label{def:c2}
	A bounded   domain  $\Og\subset \RR^{d+1}$ is called  $C^2$ if there exist numbers $\da>0$, $M>0$ and a finite cover of the boundary $\Ga:= \p \Og$  by connected open sets $\{ U_j\}_{j=1}^J$ such that: {\rm (i)}  for every  $x\in\Og$ with   $\dist(x, \Ga)<\da$,  there exists an index $j$ such that $x\in U_j$,  and $\dist(x, \p U_j)>\da$;  {\rm (ii)} for each $j$ there exists a  Cartesian coordinate system $(\xi_{j,1},\dots, \xi_{j,d+1})$ in $U_j$ such that the set $\Og\cap U_j$ can be represented by the inequality $\xi_{j,d+1}\leq f_j(\xi_{j,1}, \ldots, \xi_{j, d})$, where $f_j:\RR^{d}\to \RR$ is a $C^2$-function satisfying
	$\max_{1\leq i, k\leq d}\|\p_i\p_k f_j\|_\infty\leq M.$
	
\end{defn}

\subsection{New moduli of smoothness on $C^2$ domains}\label{modulus:def}

Let $\Og\subset \RR^{d+1}$ be  the closure of an open,   connected, bounded  $C^2$-domain  in $\RR^{d+1}$ with boundary $\Ga=\p \Og$. In this section, we shall give the definition of  our    new moduli  of smoothness on  the   domain $\Og$.

The definition requires  a tangential  modulus of smoothness $\wt{\og}^r_G (f, t)_p$ on a domain $G\subset \Og$ of special type, which is described  below.  We start with an upward  $x_{d+1}$-domain $G$   given in~\eqref{2-7-special} with  $\xi=0$. 
Let 
$$\xi_j(x):=  {e_j + \p_j g(x) e_{d+1}}\in\RR^{d+1},\  \ j=1,\dots, d,\   \  x\in (-2b, 2b)^d.$$
Clearly, $\xi_j(x)$ is the tangent vector to the essential boundary $\p' G^\ast$ of $G^\ast$  at the point $(x, g(x))$ that is parallel to the $x_jx_{d+1}$-coordinate plane.  Given  a parameter    $A_0>1$,  we  set 
\begin{equation}\label{eqn:a0}
	G^t:= \bl\{\xi\in G:\  \ \dist(\xi, \p' G) \ge A_0 t^2\br\},\   \ 0\leq t\leq 1.
\end{equation} 
We then   define 
the    $r$-th order tangential   modulus of smoothness  $\wt{\og}^r_G (f, t)_p$, ($0<t\leq 1$)  of $f\in L^p(\Og)$   by 
\begin{align}\label{modulus-special}
	\wt{\og}^r_G (f, t)_p&:= \sup_{\sub{0<s\leq t\\
			1\leq j\leq d}} \Bl(\int_{G^t}
	\Bl[\f 1{(tb)^d}  \int_{I_x(tb)} |\tr_{s \xi_j(u)}^r (f, \Og,(x,y))|^p  \, du\Br] dxdy\Br)^{\f1p},
\end{align}
where $I_x(tb):=\{ u\in (-b, b)^d:\   \  \|u-x\|\leq tb\}$, and we use $L^\infty$-norm to replace  the $L^p$-norm  when $p=\infty$.
For $t>1$, we define   $ \wt{\og}^r_G (f, t)_p=\wt{\og}^r_G (f, 1)_p$.
Next, if  $G\subset \Og$ is a general domain of special type,  then we  define the tangential  moduli $\wt{\og}^r_G (f, t)_p$ 
through the identity,
$$ \wt{\og}^r_G (f, t)_p=\wt{\og}^r_{T(G)} (f\circ T^{-1}, t)_p,$$
where $T$ is   a  composition of a  translation and the reflections $\sa_j, \tau_j$ for some $1\leq j\leq d+1$ which takes $G$ to an upward $x_{d+1}$-domain  of  the form~\eqref{2-7-special} with  $\xi=0$.

To define the  new moduli of smoothness on $\Og$, we also  need  the following covering lemma, which was proved in~\cite{Da-Pr-Bernstein}*{Section~2}.


\begin{lem}[\cite{Da-Pr-Bernstein}*{Proposition~2.7}]\label{lem-2-1-18}
	There exists a finite cover of the boundary $\Gamma=\p \Og$ by domains of special type  $G_1, \dots, G_{m_0}\subset \Og$ that are attached to $\Ga$. In addition,  we may select the domains $G_j$   in such a way that    the  size of each  $G_j$  is as small as we wish, and  the  parameter  of  each  $G_j$  satisfies  the condition~\eqref{parameter-2-9}.
\end{lem}

Now we are in a position to define the    new moduli of smoothness    on $\Og$. 
\begin{defn}\label{def:modulus}
	Given  $0<p\leq \infty$,  the $r$-th order  modulus of smoothness  of $f\in L^p(\Og)$ is defined by   
	\begin{equation}\label{eqn:defmodulus}
		\og_\Og^r(f, t)_p:=\og^r_{\Og, \vi} (f, t)_p +\og^r_{\Og, \tan} (f, t)_p,  
	\end{equation}
	where 
	$$
	\og_{\Og, \vi}^r (f,t)_p:=\max_{1\leq j\leq d+1} \og^r_{\Og, \vi} (f, t; e_j)_p\  \ \text{and}\  \ 
	\og_{\Og, \tan}^r (f, t)_p :=\sum_{j=1}^{m_0} \wt{\og}_{G_j}^r (f,t)_p.
	$$
	Here $G_1,\dots, G_{m_0}\subset 	\Og$ are  the domains of special type  from  Lemma~\ref{lem-2-1-18}.
\end{defn}

Note that  the second term on the right hand side of~\eqref{eqn:defmodulus} 
is defined via  finite differences   along certain  tangential directions of  the boundary $\Ga=\p \Og$. As a result, 
we call $\og^r_{\Og, \tan} (f, t)_p$  the tangential part of the $r$-th order modulus of smoothness on $\Og$. More specifically, in~\eqref{modulus-special}, which is the main component of the tangential modulus, for each point $(x,y)\in G_t$ where a finite difference is computed, we find the ``closest'' (measuring only along $(d+1)$-st coordinate) boundary point of the domain $(x,g(x))$ and take the direction of the tangent vector $\xi_j(x)$ from the $x_jx_{d+1}$-coordinate plane. Such directions ``follow'' (are ``parallel'' to) the boundary and allow to capture the required smoothness information from the function in the tangential directions as the $C^2$-smoothness of the boundary ensures that $(x,y)+rt\xi_j(x)\in\Omega$ when necessary. (Observe that we start with any point $(x,y)\in G_t$, so such a point is at least $A_0t^2$ away from the boundary allowing for sufficient space inside the domain for the other points of the finite difference.)  Towards the interior of the domain the job is done by $\og_{\Og, \vi}^r (f,t)_p$ which only uses the coordinate directions (there is no need to be perfectly orthogonal to the boundary) and is a rather straightforward generalization of the one-dimensional modulus for the segment. 

Comparing the above with the moduli in~\eqref{eqn:ivanov}, one can see that the point sets where the finite differences are computed in~\eqref{eqn:defmodulus} are from the local subdomains resembling $U(\xi,t)$ (see the discussion after~\eqref{eqn:ivanov} for the boundary case). However, only more specific directions of the finite differences are needed in~\eqref{eqn:defmodulus} and those directions are easily expressed through the decomposition into the domains of special type (they are $\xi_j(x)$ as well as the coordinate directions). 

The modulus from~\eqref{eqn:mod-totik} is similar in the interior (non-tangential) directions also computing the finite differences along segments. However, for the tangential directions, $\wh{\og}^r(f,\da)_\Og$ uses finite differences along arcs of circles which are inside the domain and parallel to one of the coordinate axes. It is not hard to observe that the ``size'' of such circular finite differences matches that for the linear finite differences for the other two moduli: roughly $t\approx \sin t$ in the tangential and $t^2\approx 1-\cos t$ in the interior directions near the boundary.

We conclude this subsection with the following remark.

\begin{rem}\label{rem-3-2} The moduli of smoothness defined in Definition~\ref{def:modulus} rely on the parameter $A_0$ in~\eqref{eqn:a0}.  To emphasize  the dependence on this parameter,   we  often    write 
	$$\wt{\og}^r_G (f, t; A_0)_p:=\wt{\og}^r_G (f, t)_p,\   \    \  {\og}^r_{\Og} (f, t; A_0)_p:={\og}^r_{\Og} (f, t)_p.$$	
	By the Jackson theorem (Theorem~\ref{Jackson-thm}) and the univariate Remez inequality (see \cite{MT2}), it can be easily shown that given any two parameters $A_1, A_2\ge 1$,
	$$ {\og}^r_{\Og} (f, t; A_1)_p\sim {\og}^r_{\Og} (f, t; A_2)_p,\   \  t>0,\  \  0<p\leq \infty.$$	
\end{rem}

%

\subsection{Summary of main results}\label{summary}

In this subsection, we shall summarize the main results of this paper. As always, we assume that $\Og$ is the closure of  an open,  connected and  bounded $C^2$-domain in $\RR^{d+1}$. 
For simplicity, we identify with $L^\infty(\Og)$ the space $C(\Og)$ of continuous functions on $\Omega$.

The  main aim of  this paper is to prove the   Jackson type inequality and  the corresponding  inverse inequality for  the modulus of smoothness $\og_\Og^r(f,t)_p$  defined in~\eqref{eqn:defmodulus}, as  stated   in the following two theorems.

\begin{thm}\label{Jackson-thm} If   $r,n\in\NN$, $0<p\leq \infty$ and $f\in L^p(\Og)$,   then  
	$$ E_n (f)_{L^p(\Og)} \leq C \og_\Og^r(f, n^{-1})_p,$$
	where the constant $C$ is independent of $f$ and $n$.  
	
\end{thm}

\begin{thm}\label{inverse-thm} If   $r, n\in\NN$,  $1\leq p\leq \infty$ and $f\in L^p(\Og)$,  then 
	$$\og_{\Og}^r (f, n^{-1})_{p} \leq\f{ C}{n^r}   \sum_{j=0}^n (j+1)^{r-1} E_j (f)_{L^p(\Og)},$$
	where the constant $C$ is independent of $f$ and $n$. 
\end{thm}

As an example of application of the above, we obtain the following relation between approximation and smoothness classes (for further details in the classical settings, see, for example~\cite{De-Lo}*{Sect.~2.10, 7.9, 8.7}).

\begin{cor}	
	Suppose $1\le p\le\infty$, $0<q\le\infty$ and $0<\alpha<r$. For $f\in L^p(\Omega)$, we have\\
	(i) $E_n(f)_{L^p(\Omega)}=O(n^{-\alpha})$, $n=1,2,\dots$, if and only if $\og_{\Og}^r (f, t)_{p}=O(t^\alpha)$, $t>0$.\\
	(ii) $\ds \sum_{n=1}^\infty [n^\alpha E_n(f)_{L^p(\Omega)}]^qn^{-1}<\infty$ if and only if $\ds\int_0^\infty\frac{[t^{-\alpha}\og_{\Og}^r (f, t)_{p}]^q}{t}\,dt<\infty$.
\end{cor}
An implication of this corollary is that the corresponding smoothness classes (for example, the class of functions $f\in L^p(\Omega)$ satisfying $\og_{\Og}^r (f, t)_{p}=O(t^\alpha)$, $t>0$) do not depend on the particular choices of parameters $A_0$, $L$, $b$ and the decomposition into the domains of special type. 

Note that  the Jackson inequality stated in  Theorem~\ref{Jackson-thm} holds  for the full range of $0<p\leq \infty$. 

Now let us describe  two main ingredients in the proof of the direct Jackson theorem:   multivariate    Whitney type  inequalities  on certain  domains (not necessarily convex);  and localized  polynomial partitions  of   unity on $C^2$-domains. 

The  Whitney type inequality  gives an upper estimate  for the error of local polynomial   approximation of a function via the behavior of its finite differences.
A   useful  multivariate Whitney type inequality was   established by Dekel and Leviatan  \cite{De-Le}  on a convex body (compact convex set with non-empty interior)  $G\subset \R^{d+1}$  asserting  that 
for any  $0<p\leq \infty$ ,   $r\in\NN$, and  $f\in L^p(G)$,
\begin{equation}\label{7-1-18-00}E_{r-1}  (f)_{L^p(G)} \leq C(p,d,r) \og^r (f, G)_{p}.\end{equation}
It is remarkable that the constant $C(p,d,r)$ here  depends only on the three parameters $p,d,r$, but  is  independent of  the particular shape of the convex body $G$.  However, the  Whitney inequality~\eqref{7-1-18-00}  is NOT enough for our purpose because  our domain $\Og$ is not necessarily  convex, and the definition of our local moduli of smoothness (Definition~\ref{def-8-1})  uses local  finite differences along a finite number of  directions only.
In~\cite{Da-Pr-Whitney} we developed a new method to study the following Whitney type inequality  for directional  moduli of smoothness on a more general domain $G\subset\RR^{d+1}$ (not necessarily convex):
\begin{equation*}\label{7-8-18-00}E_{r-1} (f)_{L^p(G)} \leq C  \og^r(f, G; \mathcal{E})_p.\end{equation*}
The key idea of~\cite{Da-Pr-Whitney} is to deduce   the   Whitney type inequality on a  more   complicated   domain   from  the Whitney inequality  on  cubes or some other simpler domains. We state the result from~\cite{Da-Pr-Whitney} which is sufficient for the purposes of this work in Section~\ref{sec:tools}.

 A polynomial partition of unity is a useful tool  to patch together local polynomial  approximation and can be of independent interest.  For simplicity,  we say 
a set  $\Ld$  in a metric space $(X,\rho)$   is   $\va$-separated  for some $\va>0$  if $\rho(\og, \og') \ge \va$ for any two distinct points $\og, \og'\in\Ld$,  and we call  an  $\va$-separated  subset $\Ld$  of $X$   maximal  if 
$\inf_{\og\in\Ld} \rho(x, \og)<\va$ for any $x\in X$. 
In Section~\ref{sec:partition of unity}, relying on ideas by Dzjadyk and Konovalov~\cite{Dz-Ko}, we shall prove the following  localized  polynomial partitions of  unity on $C^2$-domains:
\begin{thm}\label{polyPartition00}Given  any parameter  $\ell>1$ and  positive integer $n$,   there exist a $\f 1n$-separated subset $\{\xi_j\}_{j=1}^{m_n}$ of $\Og$  with respect to the metric $\rho_\Og$ defined in~\eqref{metric}  and  a sequence  of   polynomials  $\{P_j\}_{j=1}^{m_n}\subset \Pi_{c_0 n}^{d+1}$   such that 	$\sum_{j=1}^{m_n} P_j (\xi) =1$  and 
	$  |P_j(\xi)| \leq C_1 (1+n\rho_\Og(\xi,\xi_j))^{-\ell}$, $j=1,2,\dots, m_n$ 
	for every $\xi\in \Og$,
	where the constants $c_0$ and $C_1$ depend only on $d$ and $\ell$.  	
\end{thm}

A crucial role in the proof of our inverse theorem  (i.e., Theorem~\ref{inverse-thm}) is played by a new Bernstein inequality associated with the tangential derivatives on the boundary $\Ga$, which we recently established in~\cite{Da-Pr-Bernstein}. The corresponding definitions and statements required in the context of the current work can be found in Section~\ref{sec:tools}. We only mention here that this new tangential Bernstein inequality was used in~\cite{Da-Pr-Bernstein} to establish Marcinkiewicz-Zygmund type inequalities and positive cubature formulas on $C^2$ domains.

We also compare the  moduli of smoothness $\og_\Og^r(f,t)_p$  with 
the average $(p,q)$-moduli of smoothness  $ \tau_r (f, t)_{p,q}$ introduced by Ivanov~\cite{Iv}.  
It turns out  that the  moduli  $\og_\Og^r(f,t)_p$ 
can be controlled above by the average moduli  $ \tau_r (f, t)_{p,q}$, as shown in the following theorem that will be proved in Section~\ref{ch:IvanovModuli}:

\begin{thm}\label{thm-9-1-00}
	For any $0<q\leq p\leq \infty$ and $f\in L^p(\Og)$, 
	$$\og_\Og^r(f, t; A_0)_p \leq C \tau_r (f, c_0 t)_{p,q},\   \   0<t\leq 1,$$
	where the constant $C$ is independent of $f$ and $t$. 
\end{thm}

As an immediate consequence of Theorem~\ref{thm-9-1-00} and Theorem~\ref{Jackson-thm}, we  obtain a  Jackson type  inequality for the average moduli of smoothness  for  any dimension   $d\ge 1$ and  the full range of $0<q\leq p\leq \infty$.
\begin{cor}  \label{cor-4-4-0}	
	If $f\in L^p(\Og)$, $0< q\leq p \leq \infty$ and $r\in\NN$, then
	$$ E_n (f)_p \leq C_{r, \Og} \tau_r \Bl(f, \f {c_0} n\Br)_{p,q}.$$
\end{cor}

As mentioned  in the introduction,  Corollary~\eqref{cor-4-4-0} for $1\leq p\leq \infty$ and $d=1$  was  announced  in  \cite{Iv}  for   a piecewise $C^2$-domain $\Og\subset \RR^2$.  

We shall   prove the corresponding  inverse theorem for the average moduli of smoothness $\tau_r(f, t)_{p,q}$ as well:

\begin{thm} \label{thm-15-1-00}If $r\in\NN$, $1\leq q\leq  p\leq \infty$ and $f\in L^p(\Og)$, then 
	$$\tau_r (f, n^{-1})_{p,q} \leq C_{r} n^{-r} \sum_{s=0}^n (s+1)^{r-1} E_s (f)_p.$$
\end{thm}

In the case when $\Og\subset \RR^2$ (i.e., $d=1$),  Theorem~~\ref{thm-15-1-00}  was announced   without detailed proofs  in  \cite{Iv}  for  the case   $p=\infty$  and  the case when  $1\leq p\leq \infty$ and   $\Og$ is a parallelogram or a disk.

The rest of the paper is organized as follows. Section~\ref{sec:tools} is devoted to the statements of the required Bernstein and Whitney-type inequalities obtained in~\cite{Da-Pr-Bernstein} and~\cite{Da-Pr-Whitney}.  
Sections~\ref{sec:partition of unity}--\ref{ch:direct} contain the proof of the Jackson theorem (Theorem~\ref{Jackson-thm}). In Section~\ref{ch:IvanovModuli}, we compare our moduli of smoothness $\og_\Og^r(f,t)_p$ with the average moduli of smoothness $\tau_r(f,t)_{p,q}$. The main result of Section~\ref{ch:IvanovModuli} is stated in   Theorem~\ref{thm-9-1-00}. Finally, in Section~\ref{sec:15}, we prove the inverse theorems as stated in Theorem~\ref{inverse-thm} and  Theorem~\ref{thm-15-1-00}. 

\section{Tools}\label{sec:tools}

In this section we collect several necessary ingredients which we established recently in~\cite{Da-Pr-Bernstein} and~\cite{Da-Pr-Whitney}. A useful domain covering result Lemma~\ref{lem-2-1-18} has already been stated.

\subsection{Equivalence of different metrics}

Let $\rho_\Og:\Og\times \Og\to [0,\infty)$  be the  metric on $\Og$ given in \eqref{metric}. As in~\cite{Da-Pr-Bernstein}, we introduce another metric $\wh \rho_G$ on a domain $G$  of special type, which is equivalent to  the restriction of  $\rho_{\Og}$ on $G$ if $G\subset \Og$ is attached to  $\Ga:=\p\Og$. Let $G\subset \R^{d+1}$ be an $x_d$-upward domain with base size $b\in (0,1)$ and parameter $L>0$:  
\begin{align*}
	G:=\varsigma+\{ (x, y):\  \  x\in (-b,b)^{d},\   \  g(x)-Lb<  y\leq  g(x)\},\   \ \varsigma\in\RR^{d+1},
\end{align*}
where    $g$ is a $C^2$-function on $\RR^{d}$. Then 	$$G^\ast=\varsigma+\Bl\{ (x, y):\  \  x\in (-2b,2b)^{d},\   \  \min_{u\in [-2b, 2b]^{d}} g(u)-4Lb < y\leq  g(x)\Br\}$$   
and we define a metric $\wh\rho_G: \overline{G^\ast}\times \overline{G^\ast} \to (0,\infty)$ by 
\begin{equation}\label{rhog}
	\wh{\rho}_G(\varsigma+\xi, \varsigma+\eta):=\max\Bl\{\|\xi_x-\eta_x\|,
	\Bl|\sqrt{g(\xi_x)-\xi_y}-\sqrt{g(\eta_x)-\eta_y}\Br|\Br\}
\end{equation}
for all $ \xi=(\xi_x, \xi_y),\eta=(\eta_x,\eta_y)\in \overline{G^\ast}-\varsigma$.  
We can define the metric $\wh{\rho}_G$   on a more general  $x_j$-domain  $G\subset \RR^{d+1}$ (upward or downward) in a similar way.

We will use the following equivalence of the metric $\wh\rho_G$ and the restriction of $\rho_\Og$ on $G$ when $G\subset\Og$ is attached to $\Ga=\p \Og$. 

\begin{prop}[\cite{Da-Pr-Bernstein}*{Proposition~3.1}]\label{metric-lem} If   $G \subset \Og$  is   a domain of special type attached to $\Ga$, 
	then 
	\begin{equation*}\label{6-1-metric}\wh{\rho}_G(\xi,\eta)\sim \rho_{\Og} (\xi,\eta),\    \    \  \xi, \eta\in G\end{equation*}
	with the constants of equivalence depending only on $G$ and $\Og$.
\end{prop}

\subsection{Whitney type inequality}

\begin{defn}\label{def-3-4} Given $\xi\in\sph$, we say $G \subset \RR^{d+1}$  is  a regular   $\xi$-directional domain with parameter $L\ge 1$      if   there exists a  rotation $\pmb{\rho}\in SO(d+1)$ such that\begin{enumerate}[\rm (i)]
		\item
		$\pmb{\rho}(0,\dots, 0,1)=\xi$, and  $ G$ takes the form
		\begin{equation*}\label{3-6}
			G:=\pmb{\rho}\Bl(\{(x, y):  \  x\in D,\   g_1(x)\leq y\leq g_2(x)\}\Br),
		\end{equation*}
		where  	$D\subset \RR^{d}$ is compact    and  $g_i: D\to\RR$ are measurable;
		\item
		there exist an affine function (element of $\Pi_1^{d+1}$) $H:\RR^{d}\to\RR$ and a constant $\da>0$  such that   $S\subset G\subset S_L$, where
		\begin{align*}
			\pmb{\rho}^{-1} (S):&=\{(x,y):\  \  x\in D,\  \  H(x)-\da\leq y\leq H(x)+\da\},\\
			\pmb{\rho}^{-1}  (S_L):&= \{(x,y):\  \  x\in D,\  \  H(x)-L\da\leq y\leq H(x)+L\da\}.
		\end{align*}  	
	\end{enumerate}	
	In this case, we say $S$ is the base of $G$.
\end{defn}

For $r\in\NN$,  $0<p\leq \infty$ and a nonempty set  $\CE\subset \SS^d$, we define the directional  Whitney constant by
\begin{equation*}
	w_r(\Og;\CE)_p:= \sup\Bl \{ E_{(d+1)(r-1)}(f)_{L^p(\Og)}:\  \ f\in L^p(\Og),\   \ \og^r(f, \Og;\CE)_p\leq 1\Br\}.
\end{equation*}

We remark that the above definition differs from the corresponding definition in~\cite{Da-Pr-Whitney} by using approximation from the wider space $\Pi_{(d+1)(r-1)}^{d+1}$ instead of certain ``directional'' polynomial space $\Pi_{r-1}^{d+1}(\CE)$, see~\cite{Da-Pr-Whitney}*{Prop.~1.1(ii)}. This results in smaller Whitney constants which are subject to the same upper bound as in the next lemma which is sufficient for our purposes here. 

\begin{lem}[\cite{Da-Pr-Whitney}*{Lemma~2.5}]\label{cor-7-3}
	Let $G\subset \RR^{d+1}$ be a regular  $\xi$-directional  domain with parameter $L\ge 1$ and base $S$   as given in Definition~\ref{def-3-4} for some $\xi\in\SS^d$.  Let  $\CE\subset\SS^d$ be a set of directions containing $\xi$.  	Assume that   $K$  is  a measurable subset of $\RR^{d+1}$ such that $S\subset  K\cap G$   and $w_r(K; \CE)_p<\infty$ for some $r\in\NN$,  $0<p\leq \infty$.
	Then
	\begin{equation*}\label{3-9-a}
		w_r(G\cup K; \EEE)_p\leq C_{p,r} L^{r-1+2/p}(1+ w_r(K; \CE)_p),
	\end{equation*}
	where the constant $C_{p,r}$ depends only on $p$ and $r$.
	
\end{lem}

\subsection{Bernstein inequality} 

	If $P$ is an algebraic  polynomial of one variable of degree $\le n$, then by the   univariate Bernstein inequality  (\cite[p. 265]{De-Lo}),  we have that for any $b>0$ and $\al>1$, 
\begin{equation}\label{markov-bern}
	\Bl	\|(\sqrt{b^{-1}t} +n^{-1})^iP^{(i+j)}(t)\Br\|_{L^p([0,b], dt)} \le  C_\al  n^{i+2j}  b^{-(i+j)}
	\|P\|_{L^p([0,\al b])}.
\end{equation}

Let   $G\subset \RR^{d+1}$  be  an $x_{d+1}$-upward domain with base size $b>0$ and parameter $L\ge 1$  given by 
\begin{equation*}\label{11-2-2} G:=\Bl\{ (x, y)\in\RR^{d+1}:\  \ x\in (-b,b)^{d},\   \  g(x)-Lb< y\leq g(x)\Br\},\end{equation*}
where $g:\RR^{d}\to \RR$ is a $C^2$-function  satisfying  that $\min_{x\in [-2b, 2b]^{d}}g(x)= 4Lb$.
Denote
 for each  $\mu\in(0,2]$
\begin{align*}
	G(\mu):&=\{ (x, y):\  \ x\in (-\mu b,\mu b)^{d},\   \  g(x)-\mu L b<y\leq g(x)\}.
\end{align*}  
For $(x,y)\in G(2)$,  we define  
\begin{equation*}
	\da(x,y):=g(x)-y\   \   \text{and}\  \   \
	\vi_n(x,y) :=\sqrt{\da(x,y)} +\f 1n,\   \  n=1,2,\dots.
\end{equation*}


The  Bernstein type inequality on the domain  $G$ is formulated in terms of  certain  tangential derivatives along the essential boundary $\p' G$ of $G$, whose definition is given as follows.
For  $x_0\in [-2a, 2a]^{d}$,  let \begin{equation*}\label{11-4-0}
	\xi_j (x_0) :=e_j + \p_j g(x_0) e_{d+1},\   \ j=1,\dots, d
\end{equation*}
be the tangent vector to $\p G$ at the point  $(x_0, g(x_0))$ that is parallel to   the $x_jx_{d+1}$-coordinate plane. We  denote by 	$\p_{\xi_j(x_0)}^\ell$ the $\ell$-th order  directional derivative along  the   direction  of $\xi_j(x_0)$: 
\begin{equation*}\label{11-3}
	\p_{\xi_j(x_0)}^\ell:=(\xi_j(x_0)\cdot\nabla )^\ell=\sum_{i=0}^\ell \binom{\ell}i (\p_jg(x_0))^i \p_j^{\ell-i} \p_{d+1}^i,\   \ 
\end{equation*}
where $j=1,2,\dots, d$ and $ x_0\in [-2b,2b]^{d}.$ Thus, 
for   $(x,y)\in G$ and   $  f\in C^1(G)$,
$$  \p_{\xi_j(x)}^\ell f(x,y)=\sum_{i=0}^\ell \binom{\ell}i (\p_jg(x))^i( \p_j^{\ell-i} \p_{d+1}^i f)(x,y),\  \ 1\leq j\le d. $$

We also need to deal with  certain mixed directional derivatives.   Let $\NN_0$ denote the set of all nonnegative integers.  For $\pmb{\al}=(\al_1,\dots, \al_{d})\in \NN_0^{d}$, we set $|\pmb\al| =\al_1+\al_2+\dots+\al_{d}$, and define 
\begin{align*}
	{\mathcal{D}}_{\tan, x_0}^{\pmb{\al}}  =\p_{\xi_1(x_0)}^{\al_1} \p_{\xi_2( x_0)}^{\al_2} \dots \p_{\xi_{d}( x_0)}^{\al_{d}},\  \    \    \  x_0\in [-2b, 2b]^{d}.\label{11-4}
\end{align*}
%
%
%
%
%
%
%
%
%

Finally, we are ready to state the required result.

\begin{thm}\label{cor-11-2} \cite{Da-Pr-Bernstein}*{Corollary~5.2} Let $\ld \in (1,2]$ and $\mu>1$ be two given parameters.  If  $0<p\leq \infty$  and       $f\in\Pi_n^{d+1}$,  then for any $\pmb{\al}\in\NN_0^{d}$, and  $ i,j=0,1,\dots$,
	\begin{align*}
		\Bl\|\vi_n(\xi)^{i} &\max_{ u\in \Xi_{n,\mu,\ld}(\xi)}\Bigl|  {\mathcal{D}}_{\tan, u}^{\pmb{\al}}\partial_{d+1}^{i+j}f(\xi)\Br|\Br\|_{L^p(G; d\xi)} \le  c_\mu n^{|\pmb{\al}|+2j+i}\|f\|_{L^p(G(\ld))},	\end{align*}
	where 
	$$\Xi_{n, \mu,\ld} (\xi):= \Bl\{ u\in [-\ld b, \ld b]^{d}:\  \  \|u-\xi_x\|\leq \mu \vi_n(\xi)\Br\}, \quad \xi=(\xi_x,\xi_y).$$
\end{thm}

\section{Polynomial partitions of the unity}\label{sec:partition of unity}

\subsection{Polynomial partitions of the unity on domains of special type}
\label{sec:5}
The main purpose in this section is to construct a localized  polynomial partition of the  unity on a domain   $G\subset \RR^{d+1}$  of special type.  Without loss of generality, we may assume that $G$ is an upward $x_{d+1}$-domain given in~\eqref{2-7-special} with $\xi=0$, small base size $b>0$ and parameter $L=b^{-1}$. Namely,  	
\begin{align*}
	G:=\{ (x, y):\  \  x\in [-b,b]^{d},\   \  g(x)-1\leq  y\leq  g(x)\},
\end{align*}
where $b\in (0,(2\sqrt{d})^{-1})$ is a sufficiently small constant and  $g$ is a $C^2$-function on $\RR^d$ satisfying that $\min_{x\in [-b, b]^d} g(x)\ge 4$.

Our construction of localized polynomial partition of the  unity relies on  a partition of the domain $G$, which we now  describe.  
Given a positive integer $n$, let   $\Ld^d_n:=\{ 0, 1,\dots, n-1\}^d\subset \ZZ^d$ be an index set. 
We shall use boldface letters $\mathbf{i}, \mathbf{j},\dots$ to denote indices in the set $\Ld_n^d$.  For each $\mathbf{i}=(i_1,\dots, i_{d})\in \Ld_n^d$, define 
\begin{equation*}\label{partition-b}\Delta_{\bfi}:=[t_{i_1}, t_{i_1+1}]\times \dots \times [t_{i_{d}}, t_{i_{d}+1}]  \   \  \ \text{with}\  \    t_{i}=-b+\f {2i}n b.
\end{equation*}
Then  $\{\Delta_{\bfi}\}_{\bfi\in\Ld_n^d}$ forms a   partition of the cube $[-b,b]^d$.  
Next, let    $N:=N_n:=\ell_1 n$
and  $\al:= 1/(2\sin^2\f \pi{2\ell_1})$,  where  $\ell_1$ is  a sufficiently  large positive integer such that $\alpha$ satisfies \begin{equation}\label{5-2-18}
	\al\ge 5d\max_{x\in [-4b, 4b]^d} (|g(x)|+\max_{1\leq i, j\leq d} |\p_i\p_j g(x)|). 
\end{equation}
Let  $\{\al_j:=2\al \sin^2 (\f {j\pi}{2N})\}_{j=0}^N$ denote   the Chebyshev partition of the interval $[0, 2\al]$ of order $N$ such that  $\al_n=1$. Then  $\{\al_j\}_{j=0}^n$ forms a partition of the interval $[0,1]$.  
Finally, we  define a  partition of the domain $G$ as follows:  
\begin{align*}
	G&=\Bl\{(x,y):\  \  x\in [-b,b]^d,\   \   g(x)-y\in [0,1]\Br\} =\bigcup_{\bfi\in\Ld_n^d} \bigcup_{j=0}^{n-1} I_{\mathbf{i},j},
\end{align*}
where 
$$I_{\mathbf{i},j}:=\Bl\{ (x, y):\  \  x\in \Delta_{\bfi},\  \   g(x)-y\in [\al_{j}, \al_{j+1}]\Br\}.$$
Note that $\Ld_n^d \times \{0,\dots, n-1\}=\Ld_n^{d+1}$.

With the above notation, we have

\begin{thm}\label{strips-0}
	For any   $m\ge2$, there exists a sequence of   polynomials $\bl\{q_{\bfi, j}:\  \   (\bfi, j) \in\Ld_n^{d+1}\br\}$ of degree at most  $ C(m, d) n$ on $\RR^{d+1}$ such that $$\sum_{(\bfi, j)\in\Ld_n^{d+1}} q_{\bfi,j}(x,y)=1\   \   \  \text{for all $(x,y)\in G$},$$
	and  for each $(x,y)\in I_{\mathbf{k}, l} $ with $(\mathbf{k},l)\in\Ld_n^{d+1}$, 
	\begin{equation*}\label{strips-ineq}
	| q_{\bfi,j}(x,y)|\le \frac{C_{m,d}}{\Bl(1+\max\{ \|\bfi-\mathbf{k}\|, |j-l|\}\Br)^m}.
	\end{equation*}
\end{thm}

Theorem~\ref{strips-0}  is motivated by~\cite[Lemma~2.4]{Dz-Ko}, but  some important details of the proof were omitted there. In this section, we shall give a complete and simpler proof of the theorem. 


Recall that we write $\xi\in\RR^{d+1}$ in the form $\xi=(\xi_x, \xi_y)$ with $\xi_x\in\RR^d$ and $\xi_y\in\RR$.

\begin{rem}\label{rem-5-2}
	Recall that in~\eqref{rhog} we introduced  the following metric on the domain   $G$:  for  $\xi=(\xi_x, \xi_y)$ and  $\eta=(\eta_x, \eta_y)\in G$, 
	\begin{equation*}
		\wh{\rho}_G(\xi, \eta)=\max\Bl\{\|\xi_x-\eta_x\|,
		\Bl|\sqrt{g(\xi_x)-\xi_y}-\sqrt{g(\eta_x)-\eta_y}\Br|\Br\}.
	\end{equation*}
	It can be easily seen that    if $ \xi\in I_{\bfi, j}$ and $ \eta\in I_{\mathbf {k}, \ell}$, then 
	\begin{equation*}\label{Chapter-5-1}
		1+n\wh{\rho}_G(\xi,\eta) \sim 1+\max\{ \|\bfi-\mathbf{k}\|, |j-\ell|\}.
	\end{equation*}
	This implies that  
	\begin{equation*}\label{6-6-18}
		| q_{\bfi,j}(\xi)|\le \frac{C_{m,d}}{(1+n\wh{\rho}_G(\xi,\og_{\bfi,j}))^m},\   \   \ \forall  \xi\in G,\   \ \forall \og_{\ib, j} \in I_{\ib, j}.
	\end{equation*}
	
\end{rem}
\begin{rem}\label{rem-6-3}
	If  $r\in\NN$ and $n\ge 10 r$, then the polynomials $q_{\ib,j}$ in Theorem~\ref{strips-0}  can be chosen to be of total degree $\le n/r$. Indeed, this can be obtained by invoking   Theorem~\ref{strips-0} with $c(m,d)n/r$ in place of $n$, relabeling the indices, and setting some of the polynomials to be zero. 
\end{rem}
For the proof of Theorem~\ref{strips-0}, we need two additional   lemmas, the first of which is well known.

\begin{lem}\label{chebyshev}\cite[Theorem~1.1]{Dz-Ko}
	Given any parameter $\ell>1$, there exists a sequence of     polynomials $\{u_j\}_{j=1}^n$ of degree at most  $ 2n$ on $\RR$  such that $\sum_{j=0}^{n-1}u_j(x)=1$ for all $x\in [-1,1]$ and
	$$| u_j(\cos\t)|\leq \frac{C_{\ell}}
	{( 1+n|\t -\f {j\pi}{n}|)^{\ell}},\   \  \t\in [0,\pi],\  \ j=0,\dots,n-1.$$
\end{lem}

The second lemma gives a    polynomial partition of the  unity   associated with the partition $\{\Delta_{\bfj}:\  \ \bfj\in\Ld_n^d\}$ of  the cube $[-b,b]^d$. 

\begin{lem}\label{uniform} Given any  parameter $\ell>1$,  there exists a sequence of    polynomials $\{v_{\mathbf{j}}^d\}_{\mathbf{j}\in\Ld_n^d} $ of total  degree $\le  2dn$ on $\RR^d$ such that for all $x\in [-b, b]^{d}$,  $\sum_{\mathbf{j}\in \Ld_n^d}v_{\mathbf{j}}^d(x)= 1$   and
	$$|v_{\mathbf{j}}^d(x)|\le \f{C_{\ell,d}} { (1+n\|x-x_{\mathbf{j}}\|)^\ell},\    \  \mathbf{j}\in\Ld_n^d,
	$$
	where $x_{\bf j}$ is an arbitrary point in $\Delta_{\bf j}$.
\end{lem}
This lemma is probably well known, but for completeness, we present a proof below.
\begin{proof}
	Without loss of generality, we may assume that  $d=1$ and $b=\f12$. 
	The  general case can be deduced easily  using tensor products of polynomials in one variable.
	Let $\{u_j\}_{j=0}^{ n-1}$ be a sequence of
	polynomials of degree at most $2n$ as given in  Lemma~\ref{chebyshev} with $2\ell$ in place of $\ell$. Noticing that    for  $u\in [-1,1]$ and $v\in [-\f 12, \f12]$, 
	\begin{equation}\label{4-1}
		|u-v|\leq |\arccos u-\arccos v |\leq \pi |u-v|,
	\end{equation}
	we obtain  
	\begin{equation}\label{4-2}
		| u_j(x)|\leq\f{ C_\ell}{ (1+n|x-\cos \f {j\pi}n|)^{2\ell}},\  \ x\in \Bigl[-\f12,\f12\Bigr].
	\end{equation}
	Next, we define a sequence of polynomials $\{v_j\}_{j=0}^{n-1}$ of degree at most $2n$ on $[-\f12, \f12]$ as follows: 
	$$ v_j(x)=\sum_{i:\  \   s_{j}<\cos \f {i\pi}n\leq s_{j+1}} u_i(x),$$
	where $0\leq i\leq n-1$,  $s_0=-2$, $s_n =2$,  $s_j=t_j =-\f12 +\f {j}n$ for $1\leq j\leq n-1$, and  we define  $v_j(x)=0$ if the sum is taken over the  empty set.  Clearly,   $\sum_{j=0}^{n-1} v_j(x)=\sum_{i=0}^{n-1} u_i(x)=1$ for all $x\in [-\f12,\f12]$. 
	Furthermore, using~\eqref{4-2}, we have 
	\begin{align*}
		| v_j(x) |\leq \f {C_\ell}{ (1+n|x-s_j|)^{\ell}} \sum_{i=0}^{n-1} \f{ 1}{ (1+n|x-\cos \f {i\pi}n|)^{2\ell}}\leq  \f {C_\ell}{ (1+n|x-s_j|)^{\ell}}, 
	\end{align*}
	where the last step uses~\eqref{4-1}.
	This completes the proof.
\end{proof}

We are now in a position to prove Theorem~\ref{strips-0}.\\

\begin{proof}[Proof of Theorem~\ref{strips-0}]
	Set \[M:=d\max_{1\leq i,j\leq d}\max_{x\in[-b,b]^d} |\p_i\p_jg(x)|+1.\]
	For each  $\bfi\in\Ld_n^d$,  let  $x_{\bfi}\in \Delta_{\bfi}$ be   an arbitrarily fixed point in the cube $\Delta_{\bfi}$, and  define 
	$$f_{\bfi}( x):=g( x_{\bfi})+\nabla g( x_{\bfi}) \cdot ( x- x_{\bfi})+\f M2 \| x- x_{\bfi}\|^2.$$
	By   Taylor's theorem,  it is easily seen  that for each $x\in [-b,b]^d$,
	\begin{align}\label{5-8-aug}
		f_{\bfi}(x) -M\|x-x_{\bfi}\|^2 \leq g(x) \leq f_{\bfi}(x).
	\end{align}
	Since $0<b<(2\sqrt{d})^{-1}$, this implies that  for each $\bfi\in\Ld_n^d$,
	$$ G \subset\Bl \{ ( x,y):\  \ x\in [-b,b]^{d},\   \   0\leq f_{\bfi}( x)-y\leq M+1\Br\}.$$
	Recall that 
	$\{\al_j\}_{j=0}^N$
	is a Chebyshev partition of $[\al_0, \al_N]=[0, 2\al]$ of degree $ N=2\ell_1n$,  $\al_n=1$   and according to~\eqref{5-2-18}, $\al\ge 4M+1$. Thus, 
	\begin{align*}
		G\subset \bigcup_{\bfi\in\Ld_n^d} \bigcup_{j=0}^{N-1} \Bl \{ ( x,y):\  \ x\in \Delta_{\bfi},\   \   \al_{j}\leq f_{\bfi}( x)-y\leq \al_{j+1}\Br\}. 
	\end{align*}
	
	Next, using   Lemma~\ref{chebyshev}, we obtain   a sequence of  polynomials $\{u_j\}_{j=0}^{N-1}$of degree at most $4\ell_1 n$ on $[0, 2\al]$  such that   $\sum_{j=0}^{N-1} u_j(t) =1$ for all $t\in [0, 2\al]$,   and
	\begin{equation}\label{key-5-6}
		| u_j (t)| \leq \f {C_m}{ (1+n |\sqrt{t}-\sqrt{\al_j}|)^{4m}},\   \  t\in [0, 2M]\subset [0, \al].
	\end{equation}
	Similarly, using  Lemma~\ref{uniform},  we may  obtain  a sequence of   polynomials $\{v_{\mathbf{j}}\}_{\mathbf{i}\in\Ld_n^d} $ of total  degree $\le n$ on the cube $[-b, b]^d$  such that $\sum_{\mathbf{j}\in \Ld_n^d}v_{\mathbf{j}}( x)= 1$ for all $ x\in [-b, b]^{d}$,   and
	\begin{equation}\label{key-5-7}
		| v_{\mathbf{j}}( x)|\le \f{C_m} { (1+n\| x- x_{\mathbf{j}}\|)^{4m}},\   \ x\in [-b, b]^d.
	\end{equation}
	Define  a sequence $\{q^\ast_{\bfi,j}: \  \  \bfi\in\Ld_n^d,\   \  0\leq j\leq N-1\}$ of   auxiliary polynomials  as follows: 
	\begin{equation}\label{key-5-8-0}
		q^\ast_{\bfi,j}( x,y):=u_j(f_{\bfi}( x)-y)v_{\bfi}( x).
	\end{equation}
	It is easily seen from~\eqref{key-5-6} and~\eqref{key-5-7} that
	for each $( x, y)\in G$, 
	\begin{align}\label{5-7-chapter}
		|  q^\ast_{\bfi,j}(x, y)|\leq \f {C_m}{(1+n\|x- x_{\bfi}\|)^{4m} (1+n|\sqrt{f_{\bfi}({x})-y}-\sqrt{\al_j}|)^{4m}}.
	\end{align}
	
	We   claim that for each  $({x},y)\in G$,
	\begin{align}
		|  q^\ast_{\bfi,j}(x, y)|&\leq
		\f {C_m}{(1+n\|{x}- x_{\bfi}\|)^{2m} (1+n|\sqrt{g({x})-y}-\sqrt{\al_j}|)^{2m}}.\label{claim-4-5}
	\end{align}
	Note that~\eqref{claim-4-5}  follows directly from~\eqref{5-7-chapter} if $ 6M\| x- x_{\bfi}\|>  |\sqrt{g({x})-y}-\sqrt{\al_j}|$.
	Thus, for the proof of~\eqref{claim-4-5}, 
	it suffices to prove that  the equivalence
	\begin{equation}\label{key-5-11}
		|\sqrt{f_{\bfi}({x})-y}-\sqrt{\al_j}|\sim |\sqrt{g({x})-y}-\sqrt{\al_j}|,
	\end{equation}
	holds 
	under the assumption  
	\begin{equation}\label{key-5-13-18} 6M\| x- x_{\bfi}\|\leq  |\sqrt{g({x})-y}-\sqrt{\al_j}|. \end{equation}
	Indeed, if  $\sqrt{f_{\bfi}({x})-y}+ \sqrt{g({x})-y}\leq 2M\|{x}-{x}_{\bfi}\|$, then~\eqref{key-5-13-18} implies 
	$$\sqrt{\al_j}\ge 4M \|{x}-{x}_{\bfi}\|\ge 2 \max\{\sqrt{f_{\bfi}({x})-y},  \sqrt{g({x})-y}\},$$ 
	and hence 
	$$|\sqrt{f_{\bfi}({x})-y}-\sqrt{\al_j}|\sim \sqrt{\al_j}\sim |\sqrt{g({x})-y}-\sqrt{\al_j}|.$$
	On the other hand, if  $\sqrt{f_{\bfi}({x})-y}+ \sqrt{g({x})-y}> 2M\|{x}-{x}_{\bfi}\|$, then  by~\eqref{key-5-13-18} and~\eqref{5-8-aug}, we have 
	\begin{align*}
		&\Bl|\sqrt{f_{\bfi}({x})-y}- \sqrt{g({x})-y}\Br|=\f{|f_{\bfi}({x})-g({x})|}{\sqrt{f_{\bfi}({x})-y}+ \sqrt{g({x})-y}}\\
		&\leq \f {M\|{x}-{x}_{\bfi}\|^2}{2M\|{x}-{x}_{\bfi}\|}= \f12 \|{x}-{x}_{\bfi}\|\leq \f 1{12M}|\sqrt{g({x})-y}-\sqrt{\al_j}|,
	\end{align*}
	which in turn  implies~\eqref{key-5-11}. This completes the proof of~\eqref{claim-4-5}.

	Finally,   we define  for  $\bfi\in\Ld_n^d$,
	$$q_{\bfi,j}(x, y)=\begin{cases}
		q^\ast_{\bfi,j}(x,y),\  \  \text{ if $0\leq j\leq n-2$},\\
		\sum_{k=n-1}^{N-1} q^\ast_{\bfi,k}(x,y),\   \  \text{if $j=n-1$.}
	\end{cases}$$
	Clearly, each $q_{\bfi, j}$ is a  polynomial of degree at most $Cn$. 
	Since  for any $( x, y)\in G$ the polynomial $u_j$ in the definition~\eqref{key-5-8-0} is evaluated at the  point $f_{\bfi}( x)-y$, which lies in the interval $[0, M+1]\subset [\al_0, \al_N]$, it follows  that  for any $( x, y)\in{G}$,
	$$\sum_{\bfi\in\Ld_n^d} \sum_{j=0}^{n-1} q_{\bfi,j}(x,y)=\sum_{\bfi\in\Ld_n^d} \sum_{j=0}^{N-1} q^\ast_{\bfi,j}(x)=\sum_{\bfi\in\Ld_n^d} v_{\bfi}^d ( x) \sum_{j=0}^{n-1} u_{j} (f_{\bfi}( x) -y)=1.$$
	To complete the proof, 
	by~\eqref{claim-4-5}, it remains to estimate $q_{\bfi, j}$ for $j=n-1$.   
	Note  that for $j\ge n$,
	$$\sqrt{\al_j}-\sqrt{g({x})-y}\ge \sqrt{\al_n}-\sqrt{g({x})-y}\ge 0.$$
	Thus,  using~\eqref{claim-4-5}, and recalling that $m\ge 2$,  we obtain  that  
	\begin{align*}
		|q_{\bfi,n-1}(x)|&\leq  \f {C_m}{(1+n\|{x}- x_{\bfi}\|)^{2m} (1+n|\sqrt{g({x})-y}-\sqrt{\al_n}|)^{m}} \\
		&\qquad\qquad\qquad \cdot
		\sum_{j=n}^{N}\f 1{(1+n|\sqrt{g({x})-y}-
			\sqrt{\al_j}|)^{m}}\\
		&\leq \f {C_m}{(1+n\|{x}- x_{\bfi}\|)^{2m} (1+n|\sqrt{g({x})-y}-\sqrt{\al_n}|)^{m}}.
	\end{align*}
	This completes the proof.
\end{proof}

\subsection{Polynomial partitions of the unity on general $C^2$-domains}
\label{unity:sec}
In this section, we shall extend Theorem~\ref{strips-0} to the $C^2$-domain $\Og$. We will use the metric $\rho_\Og$ defined by~\eqref{metric}. 
Our  goal is to show the following theorem: 
\begin{thm}\label{polyPartition}Given  any $m>1$ and any positive integer $n$,   there exist a finite subset $\Ld$ of $\Og$ and  a sequence $\{\vi_\og\}_{\og\in\Ld}$ of   polynomials of degree at most $C( m) n$ on the domain $\Og$  satisfying 
	\begin{enumerate}[\rm (i)]
		\item $ \rho_{\Og} (\og,\og') \ge \f 1n$ 	for any two distinct points $\og,\og'\in\Ld$;  
		\item for every $\xi\in \Og$,	$\sum_{\og \in\Ld}  \vi_\og (\xi)=1$  and 
		\item for any  $\xi\in \Og$ and $\og\in\Ld$, 
		$$  |\vi_\og(\xi)| \leq C_m (1+n\rho_\Og(\xi,\og))^{-m}.$$  
	\end{enumerate} 	
\end{thm}

\begin{rem}\label{rem-6-2}
	Recall that for $\xi\in\Og$ and $\da>0$, we defined $ U(\xi,\da)=\{\eta\in\Og:\  \  \rho_{\Og}(\xi,\eta)\leq \da\}$.
	By \cite{Da-Pr-Bernstein}*{Corollary~3.3(i)}, we have
	$$\Bl|U\Bl(\xi, \f 1n\Br)\Br|\sim \f 1{n^{d+1}} \Bl( \f 1n +\sqrt{\dist(\xi, \Ga)}\Br),\   \   \   \xi\in\Og.$$
\end{rem}

%

\begin{proof}[Proof of Theorem~\ref{polyPartition}]
	
	For convenience, we say  a subset  $K\subset \Og$ admits a polynomial partition of the unity of degree $Cn$ with parameter $m>1$   if there exist a finite subset $\Ld \subset \Og$ and  a sequence $\{\vi_\og\}_{\og\in\Ld}$ of   polynomials of degree at most $C n$  such that $\rho_\Og(\og,\og') \ge \f 1n$ for any two distinct points $\og,\og'\in\Ld$,  $\sum_{\og \in\Ld} \vi_\og (x) =1$ for every $x\in K$ and 
	$ |\vi_\og (x)| \leq C (1+n\rho_{\Og} (x,\og))^{-m}$ for every $x\in K$ and $\og\in\Ld$, in which case    $\{ \vi_\og\}_{\og\in\Ld}$ is called  a polynomial partition of the unity of degree $Cn$ on   the set $K$.
	According to Theorem~\ref{strips-0}, Remark~~\ref{rem-5-2}, and Proposition~\ref{metric-lem},  if   $G\subset \Og$ is a domain  of special type attached to $\Ga$ or if $G=Q$ is a cube such that $4Q\subset \Og$, then for any $m>1$,  $G$ 
	admits a polynomial partition of the unity of degree $Cn$ with  parameter $m$.

	Our  proof relies on the decomposition in Lemma~\ref{LEM-4-2-18-0}. 
	Let $\{\Og_s\}_{s=1}^J$ be the sequence of subsets of $\Og$ given in Lemma~\ref{LEM-4-2-18-0}.  For $1\leq j\leq J$, let $H_j =\bigcup_{s=1}^j \Og_s$.  Assume that for some $1\leq j\leq J-1$, $H_j$ admits a polynomial partition $\{u_{\og_i}\}_{i=1}^{n_0}$   of the unity  of degree $Cn$ with parameter $m>1$. By induction and  Lemma~\ref{LEM-4-2-18-0}, it  will suffice to show that $H_{j+1}$ also  admits a polynomial partition of the unity of degree $Cn$ with   parameter $m>1$.
	For simplicity, we write $H=H_j$ and $K=\Og_{j+1}$. 
	Without loss of generality, we may assume that $K=S_{G,\ld_0}$ with $\ld_0\in (\f12, 1)$ and  $G\subset \Og$ a domain of special type attached to $\Ga$.  The case when $K=Q$ is a cube such that $4Q\subset \Og$ can be treated similarly, and in fact, is simpler.

	By Theorem~\ref{strips-0}, $G$ admits a   polynomial partition $\{u_{\og_j}\}_{j=n_0+1}^{n_0+n_1}$ of the  unity of degree $Cn$ with parameter $m>1$.
	Recall $H\cap G$ contains an open ball of radius $\ga_0\in (0,1)$.  Let  $L>1$ be  such that $\Og\subset B_L[0]$, and let 
	$\t:=\f {\ga_0}{20L} \in (0,1)$. 
	According to Lemma~\ref{lem-4-3},  there exists a polynomial $R_n$ of degree at most $Cn$ such that $0\leq R_n(\xi)\leq 1$ for $\xi\in B_L[0]$,  $1-R_n(\xi)\leq \t^n$ for $\xi \in K$ and $R_n(\xi)\leq \t^n$ for $x\in \Og\setminus  G$.
	We now define 
	$$ w_j(\xi) =\begin{cases} u_{\og_j}(\xi) (1-R_n(\xi)), &\   \  \text{if $1\leq j\leq n_0$},\\
		u_{\og_j}(\xi)R_n(\xi), &\   \  \text{if $n_0+1\leq j\leq n_0+n_1$}.\end{cases}$$
	Clearly, each $w_j$ is a  polynomial of degree at most $Cn$ on $\RR^{d+1}$. 
	Since  polynomials are analytic functions  and $H\cap G$ contains an open ball of radius $\ga_0$, it follows that  
	$$\sum_{j=1}^{n_0+n_1}
	w_j(\xi) =R_n(\xi)+1-R_n(\xi)=1,\   \  \forall \xi\in \RR^{d+1}.$$
	
	Next, we prove  that for each $1\leq j\leq n_0+n_1$, 
	\begin{equation}\label{key-6-2-1}
		| w_j(\xi)|\leq C (1+n\rho_{\Og}(\xi, \og_j))^{-m},\   \    \forall \xi\in H\cup K.
	\end{equation}
	Indeed,  if $1\leq j\leq n_0$, then for $\xi\in H$,
	$$|w_j(\xi)|\leq |u_{\og_j}(\xi)|\leq C (1+n\rho_{\Og}(\xi, \og_j))^{-m},
	$$
	whereas   for $\xi\in K\subset G$,
	\begin{align*}
		| w_j(\xi)| &\leq \t^n \|u_{\og_j}\|_{L^\infty (B_L[0])}\leq C\t ^n \Bl( \f {10L}{\ga_0}\Br)^n\|u_{\og_j}\|_{L^\infty(H\cap G)}\\
		\leq & C 2^{-n}\leq  C_m (1+n\rho_{\Og}(\xi, \og_j))^{-m} ,
	\end{align*} 
	where the second step uses Lemma~\ref{lem-4-1}.
	Similarly,  if $n_0<j\leq n_0+n_1$, then for $\xi\in G$,
	$$| w_j(\xi)|\leq |u_{\og_j}(\xi)|\leq C_m (1+n\rho_{\Og}(\xi, \og_j))^{-m},
	$$
	whereas  for $\xi\in H\setminus G$, 
	$$ | w_j(\xi)| \leq \t^n \|u_{\og_j}\|_{L^\infty (B_L[0])}\leq C\t ^n \Bl( \f {10L}{\ga_0}\Br)^n\leq C 2^{-n}\leq  C_m (1+n\rho_{\Og}(\xi, \og_j))^{-m}.$$ 
	Thus, in either case, we prove the estimate~\eqref{key-6-2-1}.
	
	Finally, 
	we write the set    $A:=\{\og_1,\dots, \og_{n_0+n_1}\}$ as  a disjoint union 
	$A=\bigcup_{\og\in\Ld} I_{\og}$, where  $\Ld$ is  a subset of $A$ satisfying that  
	$\min_{\og\neq  \og'\in\Ld} \rho_{\Og}(\og,\og')\ge \f 1n$, and
	$ I_\og\subset \{\og'\in A:\  \  \rho_{\Og} (\og, \og') \leq \f1n\}$  for each $\og\in\Ld$.  We then define
	$$\vi_\og(\xi): =\sum_{j:\  \    \og_j\in I_\og} w_j(\xi),\   \ \xi \in H\cup G,\  \  \og\in\Ld,$$
	where $1\leq j\leq n_0+n_1$.
	Clearly,  each $\vi_\og$ is a polynomial of degree at most $C n$ and 
	$$\sum_{\og\in\Ld} \vi_{\og} (\xi) =\sum_{j=1}^{n_0+n_1} w_j(\xi)=1,\   \   \  \forall \xi\in H\cup G.$$
	On the other hand, we recall that  
	$\rho_{\Og}(\og_i, \og_j) \ge \f 1n$ if $1\leq i\neq j\leq n_0$ or $n_0+1\leq i\neq j\leq n_0+n_1$.
	Thus,  by the standard volume estimates and Remark~\ref{rem-6-2} (or directly by~\cite{Da-Pr-Bernstein}*{Corollary~3.3(iii)}) 
	we have that  $\# I_\og \leq C(\Og, m)$ for each $\og \in \Ld$, where $\# I$ denotes the cardinality of a set $I$. It then follows from~\eqref{key-6-2-1}   that 
	$$|\vi_\og (\xi) |\leq C (1+n\rho_\Og (\xi,\og))^{-m},\   \  \xi\in H\cup G,\   \  \og \in\Ld.$$
	Thus, we have shown that  the set $H\cup K$ admits a polynomial partition of the unity of degree $cn$ with parameter $m$, completing     the induction. 
\end{proof}

\begin{rem}
	The above proof implies $\#\Lambda=O(n^{d+1})$; recall that $\Omega\subset\R^{d+1}$.
\end{rem}

%
	%
	%
%

\section{Geometric reduction near the boundary}\label{sec:reduction}

Our main goal in this section is to   show that    the Jackson inequality in   Theorem~\ref{Jackson-thm} can be deduced from  the following     Jackson-type estimates on domains  of special type.

\begin{thm}\label{THM-4-1-18} 	
	If $0<p\leq \infty$, $c>0$ is arbitrary fixed, and   $G\subset \Og$ is   an upward or downward $x_j$-domain  attached to $\Ga$ for some $1\leq j\leq d+1$,   then 
	\begin{equation*}\label{Jackson:special}
		E_n (f)_{L^p (G)}\leq C \Bl[\wt{\og}_{G}^r \Bl(f, \f c n\Br)_p+ \og_{\Og,\vi}^r \Br(f, \f cn;  e_j\Br)_{p}\Br],\   
	\end{equation*}
	where the constant $C$ is independent of $f$ and $n$.  
\end{thm}


The proof of Theorem~\ref{THM-4-1-18} will be given in Section~\ref{Sec:8}. In this section, we  will  show how Theorem~\ref{Jackson-thm} can be deduced from Theorem~\ref{THM-4-1-18}. The idea of our proof is close to that in   \cite[Chapter 7]{To17}.  

\subsection{Lemmas and geometric reduction}

We need  a series of  lemmas, the first of which gives a well known Jackson type estimate (see~\cite[Theorem~1.1]{Di96}) on a rectangular box (recall that we always assume that the sides of such boxes are parallel to the coordinate axes).

\begin{lem}\label{lem-6-1-0}
	Let $B$ be a compact  rectangular box  in $\RR^{d+1}$.  Assume that  $f\in L^p(B)$ if  $0<p< \infty$ and $f\in C(B)$ if $p=\infty$.
	Then for $0<p\leq \infty$, 
	$$\inf_{P\in \Pi_n^{d+1}} \|f-P\|_{L^p(B)} \leq C \max_{1\leq j\leq d+1}\og_{B,\vi}^r \Bl(f, \f 1n, e_j\Br)_{p},$$
	where $C$ is independent of $f$.
\end{lem}

Our second lemma is  a simple observation on domains of special type. 
Recall that  unless otherwise stated we  always assume that the parameter $L$
of a domain  of special type   satisfies the condition~\eqref{parameter-2-9}. 

\begin{lem}\label{lem-6-2:Dec} Let $G\subset \Og$ be an (upward or download) $x_j$- domain  of special type attached to $\Ga$ for some $1\leq j\leq d+1$. Then for  each parameter $\mu\in (\f12, 1]$, there exists an open rectangular box $Q_\mu$ in $\RR^{d+1}$ such that 
	\begin{equation}\label{eqn:decomp} \p' G(\mu) \subset S_{ G,\mu}:=Q_\mu\cap \Og \subset G(\mu)\   \ \text{and}  \  \  \overline{Q_\mu}\subset Q_1\   \  \text{provided $\mu<1$}.   \  \end{equation}
\end{lem}

\begin{proof}
	Without loss of generality, we may assume that $G$ is given in~\eqref{2-7-special} with $\xi=0$.  Let
	$g_{\max}:=\max_{x\in [-b, b]^d} g(x)$ and $g_{\min} :=\min_{x\in [-b, b]^d} g(x)$.
	Using~\eqref{parameter-2-9}, we have 
	$$g_{\max}-g_{\min} \leq 2 \sqrt{d} b \max_{x\in [-b, b]^d}\|\nabla g(x)\|\leq \f 12 Lb.$$
	Thus, given each parameter $\mu\in (\f 12, 1]$, we may find a constant $a_{1,\mu}$ such that
	$$ g_{\max} -\mu L b < a_{1,\mu} <g_{\min}.$$
	We may choose the constant $a_{1,\mu}$ in such a way that $a_{1,1}<a_{1,\mu}$ if $\mu<1$. 
	On the other hand, since $G$ is attached to $\Ga$, we may find an open rectangular box $Q$ of the form $(-2b, 2b)^d\times (a_1, a_2)$ such that $G^\ast =Q\cap \Og$, where $a_1, a_2$ are two constants and    $a_2>g_{\max}$. 
	Let  $a_{2, 1} =a_2$ and let  $a_{2,\mu}$ be a constant so  that $g_{\max} < a_{2,\mu} <a_2$ for $\mu\in (\f12, 1)$.   Now  setting 
	\begin{equation*}\label{6-3-Dec}
		Q_\mu :=(-\mu b, \mu b)^d\times (a_{1,\mu}, a_{2,\mu})\   \   \ \text{and}\   \ S_{G,\mu}: = Q_\mu \cap \Og,
	\end{equation*}
	we obtain~\eqref{eqn:decomp}.
\end{proof}

\begin{rem}
	Note that~\eqref{eqn:decomp} implies that $\proj_j (Q_\mu) = \proj_j (G(\mu))$ for $\mu\in (\f 12, 1]$, where $\proj_j$ denotes the orthogonal projection onto the coordinate plane $x_j=0$.
\end{rem}

Now let $G_1,\dots, G_{m_0}\subset \Og$ be the domains of special type   in Lemma~\ref{lem-2-1-18}. 
Note that  for  every  domain $G$ of special type, its essential boundary  can be expressed as 
$\p ' G=\bigcup_{n=1}^\infty \p' G (1-n^{-1})$.  
Since $\Ga$ is compact and each $\p' G_j$ is open relative to the topology of $\Ga$, there exists $\ld_0\in (\f 12, 1)$ such that $\Ga =\bigcup_{j=1}^{m_0} \p' G_j (\ld_0)$. 
For convenience, we call $S\subset \Og$ an admissible subset of $\Og$ if either $S=S_{G_j, \ld_0}$ for some $1\leq j\leq m_0$ or $S$ is an open cube in $\RR^{d+1}$ such that $4S\subset \Og$. 

Our third lemma  gives a useful decomposition of the domain $\Og$.

\begin{lem}\label{LEM-4-2-18-0}  There exists  a sequence $\{\Og_s\}_{s=1}^J$ of admissible subsets of $\Og$ such that   $\Og=\bigcup_{j=1}^J \Og_j,$ and 
	$\Og_s \cap \Og_{s+1}$ contains an open ball of radius $\ga_0>0$ in $\RR^{d+1}$ 	for  each  $s=1,\dots, J-1$, where 
	the  parameters $J$ and $\ga_0$ depend only on the domain $\Og$. 	
\end{lem}

To state the fourth  lemma,  let  $\{\Og_s\}_{s=1}^{J}$  be the sequence of sets     in Lemma~~\ref{LEM-4-2-18-0}, and let   $H_m:=\bigcup_{j=1}^m \Og_j$ for $m=1,\dots, J$.
For $1\leq j\leq J$, define $\wh{\Og}_{j}=G_i$ if $\Og_j=S_{G_i, \ld_0}$ for some $1\leq i\leq m_0$;  and $\wh{\Og}_j =2Q$ if $\Og_j$ is an open  cube $Q$ such that $4 Q\subset \Og$.

\begin{lem}\label{REDUCTION}   
	If $0<p\leq \infty$ and $1\leq j<J$, then there exist constants $c_0,C>1$ depending only on $p$ and $\Og$  such that     
	$$ E_{c_0n} (f)_{L^p(H_{j+1})} \leq C \max \Bl\{ E_n (f)_{L^p(\wh{\Og}_{j+1})}, \   E_n (f)_{L^p (H_j)}\Br\}. $$
\end{lem}

We also need a technical inequality which directly follows from the definition~\eqref{2-3-DT} and from the growth properties of the one-dimensional Ditzian-Totik modulus~\cite{Di-To}*{(4.1.3), p.~38} and~\cite{DiHI}*{(5.7)}. For any fixed $c>0$ 
\begin{equation}\label{eqn:DT-mod-growth}
	\omega^r_{\Omega,\varphi}(f,t)_p\le C \omega^r_{\Omega,\varphi}(f,ct)_p, \quad t>0,
\end{equation}
where $C$ is independent of $f$ and $t$.

Now we take Theorem~\ref{THM-4-1-18}, Lemma~\ref{LEM-4-2-18-0} and Lemma~\ref{REDUCTION}  for granted and proceed with the proof of  Theorem~\ref{Jackson-thm}.

\begin{proof}[Proof of Theorem~\ref{Jackson-thm}]
	Applying    Lemma~\ref{REDUCTION} $J-1$ times and recalling $H_J =\Og$, we obtain 
	\begin{equation}\label{6-3-0}E_{c_1n} (f)_{L^p(\Og)} \leq  C \max_{1\leq j\leq J} E_n (f)_{L^p (\wh{\Og}_j)},\end{equation}
	where $C, c_1>1$ depend only on  $p$ and $\Og$. 
	If ${\Og}_j =S_{G_i,\ld_0}$ for some $1\leq i\leq m_0$, then $\wh{\Og}_j =G_i$, and by  Theorem~\ref{THM-4-1-18}, 
	$$E_n (f)_{L^p(\wh{\Og}_j)} \leq \max_{1\leq i\leq m_0} E_n (f)_{L^p(G_i)}\leq C \og^r_\Og \left(f, \frac{c}{c_1n}\right)_p.$$
	If  $\Og_j=Q$ is a cube such that $4Q\subset \Og$, then $\wh{\Og}_j =2Q$ and by Lemma~\ref{lem-6-1-0},
	$$E_n (f)_{L^p(2Q)}\leq C \max_{1\leq j\leq d+1}\og^r_{2Q,\vi} (f, n^{-1}; e_j)_p\leq C \og_{\Og,\vi}^r (f, n^{-1})_p,$$
	where the last step uses the fact that for any   $S\subset \Og$,  
	\begin{equation*}\label{4-3-0-18}
		\max_{1\leq j\leq d+1}\og^r_{S,\vi} (f, t;  e_j)_{p}\leq C  \og_{\Og, \vi} ^r (f, t)_p.
	\end{equation*}
	By~\eqref{eqn:DT-mod-growth}, $\og_{\Og,\vi}^r (f, n^{-1})_p\le C\og_{\Og,\vi}^r (f, c/(c_1n))_p$, thus, in either case, we have 
	$$ E_n(f)_{L^p(\wh{\Og}_j)}\leq C \og_{\Og}^r \left(f, \frac{c}{c_1n}\right)_p.$$
	Theorem~\ref{Jackson-thm} then  follows  from the estimate~\eqref{6-3-0}.
\end{proof}

\begin{rem}
	\label{rem:constant-in-jackson}
	It is clear from the proof that a slightly stronger version of Theorem~\ref{Jackson-thm} is true. Namely, for arbitrary $c>0$, under the same hypotheses we obtain $E_n (f)_{L^p(\Og)} \leq C \og_\Og^r(f, c/n)_p$, where $C$ depends only on $\Omega$, $r$, $p$ and $c$. While it would be desirable to simply use the growth condition of the type~\eqref{eqn:DT-mod-growth} directly for our modulus $\og_\Og^r(f, t)_p$, it appears that establishing an analog of~\eqref{eqn:DT-mod-growth} for the tangential component of~$\og_\Og^r(f, t)_p$ is not immediate. We hope to obtain this in a future work.
\end{rem}

To  complete the reduction argument in this section, it remains  to prove  Lemma~\ref{LEM-4-2-18-0}  and    Lemma~\ref{REDUCTION}.

\subsection{Proof of Lemma~\ref{LEM-4-2-18-0}}
The proof of  Lemma~\ref{LEM-4-2-18-0}   is inspired by~\cite[p.~17]{To14} but written  in somewhat different language.
Let $S_j =S_{G_j, \ld_0}$ for $1\leq j\leq m_0$.  Note that  $S_j$   is an  open neighborhood of $\p' G_j(\ld)$  relative to the topology of $\Og$.
Since  $\p' G_j(\ld_0) \subset S_j\subset \Og$, and  $S_j$   is   open relative to the topology of $\Og$ for each $1\leq j\leq m_0$ ,  there exists $\va>0$ 
such that  $$ \Ga_\va:=\{ \xi\in\Og:\  \ \dist (\xi,\Ga)<16\sqrt{d+1}\va\}\subset \bigcup_{j=1}^{m_0} S_j.$$
Let us cover the remaining set $\Og\setminus \Ga_{\va}$ by finitely many  open cubes $Q_j$,  $j=m_0~+~1,\dots,  M_0$ of side length $\va$ such  that $4 Q_j\subset \Og$ for each $j$.
Thus, setting $E_j=S_j$ for $1\leq j\leq m_0$, and $E_j=Q_j$ for $m_0<j\leq M_0$, 
we have
$ \Og=\bigcup_{j=1}^{M_0} E_j.$ The required sets $\Omega_s$, $s=1,\dots,J$, will be selected from the family of the sets $\{E_j\}_{j=1}^{M_0}$, with possibly choosing the same set multiple times, so that each intersection $\Omega_s\cap\Omega_{s+1}$, $s=1,\dots,J-1$, contains a non-empty open ball.

First, note that if $E_j\cap E_{j'}\neq \emptyset$ for some  $1\leq j,  j'\leq M_0$, then $E_j\cap E_{j'}$ must contain a nonempty open ball in $\RR^{d+1}$.  Indeed, 
since  $E_j\cap E_{j'}$ is open relative to the topology of $\Og$, there exists an open set $V$ in $\RR^{d+1}$ such that $V\cap \Og=E_j\cap E_{j'}\neq \emptyset$. 
Since $\Og$ is the closure of an open set in $\RR^{d+1}$,  the set $V\cap \Og$ must contain   an interior point of $\Og$.

Next, we set  $\mathcal{A}=\{E_1,\dots, E_{M_0}\}$.  We say   two sets $A, B$ from the collection $\mathcal{A}$ are connected with each other  if there exists  a sequence of distinct sets $A_1, \dots, A_n$ from the collection $\mathcal{A}$  such that $A_1=A$, $A_n=B$ and $A_i\cap A_{i+1}\neq \emptyset$ for $i=1,\dots, n-1$, in which case we  write $[A:B]=\bigcup_{j=1}^n A_j$ and $(A: B)=\bigcup_{j=2}^{n-1} A_j$.
We  claim   that every set in the collection $\mathcal{A}$ is connected with the set $E_1$.  Once this claim is proved, then  Lemma~\ref{LEM-4-2-18-0} will follow since 
\begin{align*}\label{connected graph}
	\Og=\bigcup_{j=1}^{M_0} E_j = [E_1: E_2] \cup (E_2: E_1)\cup[E_1: E_3]\cup (E_3:E_1)\cup   \dots \cup [E_1: E_{M_0}]. 
\end{align*}

To show the claim, 
let  $\mathcal{B}$ denote  the collection of all    sets $E_j$ from the collection $\mathcal{A}$   that are connected with $E_1$. 
Assume that $\mathcal{A}\neq \mathcal{B}$. We obtain a contradiction as follows.  Let  $H:=\bigcup_{E\in\mathcal{B}} E$.  Then a set $E$ from the collection $\mathcal{A}$ is connected with $E_1$ (i.e., $E\in\mathcal{B}$)  if and only if $E\cap H\neq \emptyset$.
Since $\mathcal{A}\neq \mathcal{B}$,   there exists $E\in\mathcal{A}$ such that $E\cap H=\emptyset$, which
in particular, implies  that  $H$ is a proper subset of $\Og=\bigcup_{A\in \mathcal{A}} A$.
Since $\Og$ is a connected subset of $\RR^{d+1}$, $H$ must have  nonempty boundary  relative to the topology of $\Og$. Let $x_0$ be a boundary point of $H$ relative to the topology of $\Og$. Since $H$ is open relative to  $\Og$,  $x_0\in \Og\setminus H=\bigcup_{A\in\mathcal{A}} A \setminus H$.  Let $A_0\in\mathcal{A}$ be   such that    $x_0\in A_0$. Then   $A_0$ is an open neighborhood of  $x_0$ relative to the topology of $\Og$, and hence   $A_0\cap H\neq \emptyset$, which  in turn implies   $A_0\in \mathcal{B}$ and $A_0\subset H$. But this is impossible as $x_0\notin H$.

\subsection{ Proof of  Lemma~\ref{REDUCTION}}

We now  turn to the proof of  Lemma~\ref{REDUCTION}.  The proof relies on three additional  lemmas. 
The first one  is similar to~\cite[Lemma~14.3]{To14}, however, we could not follow the conclusion of its proof in~\cite{To14}, where some averaging argument appears to be missing. Our proof below uses a multivariate Nikol'skii inequality which simplifies the transition to the multivariate case.
\begin{lem}\label{lem-4-1}
	If  $B$  is  a  ball in $\RR^{d+1}$ and $\ld>1$, then  for each $P\in\Pi_n^{d+1}$ and $0<q\leq \infty$,
	\begin{equation}\label{eq1}
		\|P\|_{L^q (\ld B)} \leq C_{d,q}(5 \ld)^{n+\f {d+1}q} \|P\|_{L^q (B)}.
	\end{equation}
\end{lem}
\begin{proof} By dilation and translation, we may assume that $B=B_1[0]$.~\eqref{eq1} with the   explicit constant  $(4\ld)^n$   was proved  in~\cite[Lemma~4.2]{To14} for $q=\infty$. 
	For $q<\infty$, we have 	
	\begin{align*}
		\|P\|_{L^q (\ld B)} &\leq C_d \ld^{\f{d+1}q} \|P\|_{L^\infty (\ld B)}\leq C_d \ld^{\f {d+1}q} (4\ld)^n \|P\|_{L^\infty (B)}\\
		&\leq C_{d,q} \ld^{\f {d+1}q} (4\ld)^n n^{\f {d+1}q} \|P\|_{L^q (B)}
		\leq C_{d,q}  (5\ld)^{n+\f {d+1}q} \|P\|_{L^q (B)},
	\end{align*}
	where we used H\"older's inequality in the first step,~\eqref{eq1}  for the already proven case $q=\infty$ in the second step, and Nikol'skii's  inequality for algebraic polynomials on the unit ball (see~\cite{Da06} or \cite[Section~7]{Di-Pr16}) in the third step. 
\end{proof}

The second lemma  is probably well known. It can be proved in the same way as in~\cite[Lemma~4.3]{To14}.
\begin{lem}\label{lem-4-2} Let $I$ be a parallelepiped in $\RR^d$.  Then given parameters  $R>1$ and $\theta,\mu\in (0,1)$, there exists a polynomial  $P_n$ of  degree at most $ C(\theta, \mu, R, d) n$  such that  $ 0\leq P_n(\xi)\leq 1$ for  $ \xi\in B_R[0]$,  
	$ 1-P_n(\xi) \leq \theta^n$ for  $\xi\in \mu I$, and 
	$P_n(\xi)\leq \theta^n$ for $\xi\in  B_R[0]\setminus I$,
	where  $\mu I$ denotes the dilation of $I$  from its center by a factor $\mu$.
\end{lem}

As a consequence of Lemma~\ref{lem-4-2}, we have

\begin{lem}\label{lem-4-3}    Let   $G\subset \Og$  be   a domain of special type attached to $\Ga$, and $S_{G, \mu}: =\Og \cap Q_\mu$ be  as defined in Lemma~\ref{lem-6-2:Dec} with  $\mu\in (\f 12, 1]$. Let $R\ge 1$ be such that $Q_1 \cup \Og \subset B_R[0]$. 
	Then given   $\ld\in (\f 12, 1)$ and $\t \in (0,1)$, there exists a polynomial $P_n$ of degree at most $C(d,  \t,  R,  G, \ld)
	n$ with the properties that  $0\leq P_n(\xi)\leq 1$ for $\xi\in B_R[0]$,  $1-P_n(\xi)\leq \t^n$ for $\xi \in S_{G,\ld}$ and $P_n(\xi)\leq \t^n$ for $\xi\in \Og\setminus  S_{G,1}$.	\end{lem}

\begin{proof}
	Since $\ld<1$ and  $Q_\ld$ is an open  rectangular box such that  $\overline{Q_\ld}\subset Q_1$, it follows by Lemma~\ref{lem-4-2} that  there exists a  polynomial $P_n$  of degree at most $Cn$ such that 
	$0\leq P_n(\xi)\leq 1$ for all $\xi\in B_{R} [0]$,  $1-P_n(\xi) \leq \ta^n$ for all $\xi\in Q_\ld$ and $P_n(\xi) \leq \ta^n$ for all $\xi\in B_R[0]\setminus Q_1$.
	To complete the proof, we just need to observe that 
	$$ \Og \setminus S_{G,1} =\Omega \setminus (Q_1 \cap \Og)=\Omega\setminus Q_1\subset B_R[0]\setminus Q_1.$$
\end{proof}

We are now in a position to prove Lemma~\ref{REDUCTION}.

\begin{proof}[Proof of Lemma~\ref{REDUCTION}]

	The proof is essentially a repetition of that of~\cite[Lemma~4.1]{To14} or~\cite[Lemma~3.3]{To17} for our situation.   Let $R>1$ be such that $\Og\subset B_R[0]$, and set  $\theta:=\min\{\frac{\ga_0}{5R}, \f12\}$.   Write $H=H_j$ and  $S=\Og_{j+1}$.
	Without loss of generality, we may assume that $S=S_{G,\ld_0}$  for some  domain $G$ of special type attached to $\Ga$. (The  case  when  $S$ is a cube $Q$ such that $4Q\subset \Og$ can be proved similarly using Lemma~\ref{lem-4-2} instead of Lemma~\ref{lem-4-3}).  Then $S_{G,\ld_0}\cap H$ contains a ball $B$  of radius $\ga_0$. 	 
	By  Lemma~\ref{lem-4-3}, there exists  a polynomial $R_n$ of degree $\leq C(d, R, G)n$ such that
	$0\leq R_n(x)\leq 1$ for all $x\in B_{R}[0]$, $R_n(x)\leq \theta^{-n}$ for $x\in \Og\setminus  S_{G,1}$ and $1-R_n(x)\leq \theta^{-n}$ for $x\in  S_{G,\ld_0}$.
	Let 
	$P_1, P_2\in\Pi_n^{d+1}$ be such that 
	$$ E_n(f)_{L^p(S_{G,1})} =\|f-P_1\|_{L^p(S_{G,1})}\   \  \text{and}\  \   \   E_n(f)_{L^p(H)} =\|f-P_2\|_{L^p(H)}.$$
	Define  
	$$ P(x):=R_n(x) P_1(x) +(1-R_n(x)) P_2(x)\in \Pi_{cn}^{d+1}.$$
	Then  
	\begin{align*}
		E_{cn} (f)_{L^p(H_{j+1})}&\leq 
		\|f-P\|_{L^p(H\cup S_{G,\ld_0})}\\
		& \leq  \|f-P\|_{L^p(H\cap S_{G,1})}+\|f-P\|_{L^p(H\setminus S_{G,1})}+\|f-P\|_{L^p(S_{G,\ld_0})}.
	\end{align*}  
	
	First, we can estimate the term $\|f-P\|_{L^p(H\cap S_{G,1})}$ as follows: 
	\begin{align*}
		\|f-P\|_{L^p(H\cap S_{G,1})}&=\|R_n (f-P_1) +(1-R_n) (f-P_2)\|_{L^p(H\cap  S_{G,1})}\\
		&\leq C_p \max\Bl\{\|f-P_1\|_{L^p(S_{G,1})}, \  \|f-P_2\|_{L^p(H)}\Br\}\\
		&\leq C_p \max \Bl\{ E_n (f)_{L^p(S_{G,1})}, E_n(f)_{L^p(H)} \Br\}.
	\end{align*}
	
	Second, we show  
	\begin{align}\label{6-7-00}
		\|f-P\|_{L^p(H\setminus S_{G,1})}\leq  & C_{p,\ga_0,R} \max \Bl\{ E_n (f)_{L^p(S_{G,1})}, E_n(f)_{L^p(H)} \Br\}.
	\end{align}
	Indeed,   we have 
	\begin{align}
		\|f-P\|_{L^p(H\setminus S_{G,1})}&=\|( f-P_2) + R_n (P_2-P_1)\|_{L^p(H\setminus S_{G,1})}\notag\\
		&\leq C_p  E_n(f)_{L^p(H)}+ C_p \theta^{n} \|P_1-P_2\|_{L^p(\Og)}.\label{second}\end{align}
	However, by  Lemma~\ref{lem-4-1}, 
	\begin{align}
		\|P_1-P_2\|_{L^p (\Og)}& \leq \|P_1-P_2\|_{L^p (B_{R}[0])}\leq C\Bl( \f {5 R}{\ga_0}\Br)^{n+\f {d+1}p} \|P_1-P_2\|_{L^p(B)}\notag\\
		&\leq C\Bl( \f {5 R}{\ga_0}\Br)^{n+\f {d+1}p} \|P_1-P_2\|_{L^p(H\cap S_{G,1})}\notag\\
		&\leq C(R, d, \ga_0,p) \theta^{-n} \max\Bl\{E_n(f)_{L^p(S_{G,1})}, \  E_n(f)_{L^p(H)}\Br\}.\label{3-2-eq}
	\end{align}
	Thus, combining~\eqref{second} with~\eqref{3-2-eq}, we obtain~\eqref{6-7-00}.
	
	Finally, we estimate the term $\|f-P\|_{L^p(S_{G,\ld_0})}$  as follows: 
	\begin{align*}
		\|f-P\|_{L^p(S_{G,\ld_0})} & =\|f-P_1 +(1-R_n) (P_1-P_2)\|_{L^p(S_{G,\ld_0})} \\
		&\leq C_p \|f-P_1\|_{L^p(S_{G,1})}  + C_p \theta^{n} \|P_1-P_2\|_{L^p (\Og)}\\
		&\leq C_{p,\ga_0,R}\max\Bl\{E_n(f)_{L^p(S_{G,1})}, \  E_n(f)_{L^p(H)}\Br\},
	\end{align*}
	where the last step uses~\eqref{3-2-eq}.
	
	Now putting the above estimates together, and noticing $S_{G,1} \subset G=\wh{\Og}_{j+1}$, we  complete the proof of  Lemma~\ref{REDUCTION}.
\end{proof}

\section{The  direct Jackson theorem}\label{ch:direct}

\subsection{Jackson inequality on domains of special type}\label{Sec:8}
We will first prove  the Jackson inequality,  Theorem~\ref{THM-4-1-18}, on  a domain $G$ of special type that is attached to $\Ga=\p \Og$. Without loss of generality,
we may assume that
\begin{align}\label{standard}
	G:=\{ (x, y):\  \  x\in (-b,b)^{d},\   \  g(x)-1\le  y\leq  g(x)\},
\end{align}
where $b\in (0,(2\sqrt{d})^{-1})$ is the base size of $G$,  and  $g$ is a $C^2$-function on $\RR^d$ satisfying that $\min_{x\in [-4b, 4b]^d} g(x)\ge4$. We may choose the base size $b$ to be sufficiently small so that 
\begin{equation}\label{8-1-18}
	\max_{x\in [-4b, 4b]^d}  \|\nabla g(x)\|\leq \f 1{200db}
	\quad\text{and}\quad  \|\nabla^2 g\|_{L^\infty ([-b, b]^d)}\le \frac1{1600 b^2}.
\end{equation}

We first   recall some  notations from Section~\ref{sec:5} and Section~\ref{modulus:def}.
Given  $n\in\NN$,  the    partition $\{\Delta_{\bfi}\}_{\bfi\in\Ld_n^d}$ of the cube $[-b,b]^d$ is  defined  by 
\begin{equation*}\label{partition-a}\Delta_{\bfi}:=[t_{i_1}, t_{i_1+1}]\times \dots \times [t_{i_{d}}, t_{i_{d}+1}]  \   \  \ \text{with}\  \    t_{i}=\Bl(-1+\f {2i}n\Br)b,
\end{equation*}
where 
$\Ld_n^d:=\{ 0, 1,\dots, n-1\}^d\subset \ZZ^d$ is the  index set.  For simplicity, we also set  $t_i=-b$ for  $i<0$, and $t_i =b$ for $i>n$, and therefore, $\Delta_{\ib}$ is defined for all $\ib\in\ZZ^d$.
Next,  the sequence,
\begin{equation}\label{8-3-0-18}
	\al_j:=2\al \sin^2 \Bl(\f {j\pi}{2N}\Br),\   \  j=0,1,\dots, N:=2\ell_1 n, 
\end{equation}
forms a Chebyshev partition of the interval $[0, 2\al]$, where 
$\al:= 1/(2\sin^2\f \pi{2\ell_1})$,  and  $\ell_1$ is   a fixed  large  positive integer  for which~\eqref{5-2-18} is satisfied. 
Note that $\al_n=1$, and 
\begin{align}\label{8-4-18-0}
	\f {4j\al} {N^2} \leq \al_j-\al_{j-1} \leq \f { \pi^2 j\al} {N^2},\   \  j=1,\dots, N.
\end{align}
Finally, a  partition of the domain $G$ is  defined  as   
\begin{align*}
	G&=\Bl\{(x,y):\  \  x\in [-b,b]^d,\   \   g(x)-y\in [0,1]\Br\} =\bigcup_{(\bfi,j)\in\Ld_n^{d+1}} I_{\mathbf{i},j},
\end{align*}
where 
$$I_{\mathbf{i},j}:=\Bl\{ (x, y):\  \  x\in \Delta_{\bfi},\  \   g(x)-y\in [\al_{j}, \al_{j+1}]\Br\}.$$

Next,  we introduce a few  new notations for this section.  Without loss of generality, we assume that $n\ge 50$. By~\eqref{8-1-18}, we can select $10\leq m_0, m_1\leq n/5$ to be two fixed large integer parameters  satisfying 
\begin{equation}\label{8-5-18}
	m_1\ge \f {32\ell_1^2 m_0^2 b^2}{\al} \|\nabla^2 g\|_{L^\infty ([-b, b]^d)}.
\end{equation}
We define,  for $\bfi\in \Ld_n^d$,  
$$\Delta_{\bfi}^\ast =[t_{i_1-m_0}, t_{i_1+m_0}]\times [t_{i_2-m_0}, t_{i_2+m_0}]\times \dots\times [t_{i_{d}-m_0}, t_{i_{d}+m_0}],$$
and for $(\ib, j) \in\Ld_n^{d+1}$, 
$$I_{\bfi,j}^\ast:=\Bl\{ (x, y):\  \ x\in \Delta_{\bfi}^\ast,\   \  \al^\ast_{j-m_1}\leq g(x)-y\leq \al^\ast_{j+m_1}\Br\},$$
where
$\al_j^\ast =\al_j$ if $0\leq j\leq n$,  $\al_j^\ast =0$ if $j<0$ and $\al_j^\ast =1$ if $j>n$.  
Let $x_{\bfi}^\ast$ be   
an arbitrarily given  point    in the set $ \Delta_{\bfi}^\ast$.   Denote by   $\zeta_{k}(x_{\bfi}^\ast)$  the unit tangent vector to the boundary $\Ga$ at the point 
$({x}^\ast_{\bfi}, g({x}^\ast_{\bfi}) )$  that is  parallel to the   $x_kx_{d+1}$-plane and satisfies $\zeta_{k} (x_{\bfi}^\ast)\cdot e_k>0$ for $k=1,\dots, d$; that is,  
$\zeta_{ k} (x_{\bfi}^\ast):= \f { e_k + \p_k g( x_{\bfi}^\ast) e_{d+1} }{\sqrt{1+|\p_k g( x_{\bfi}^\ast)|^2}}.$
Set
$$\mathcal{E}(x_{\bfi}^\ast):=\{\zeta_{1} (x_{\bfi}^\ast), \dots, \zeta_{d} (x_{\bfi}^\ast)\},\   \  \ib\in\Ld_n^d.$$
By Taylor's theorem, we have 
\begin{equation*}\label{8-9-18}
	\Bl|g(x) -H_{\bfi}(x)  \Br| \leq M_0 n^{-2},\   \  \forall x\in \Delta_{\bfi}^\ast,
\end{equation*}
where     $$H_{\bfi} (x):=g(x_{\bfi}^\ast)+ \nabla g(x_{\bfi}^\ast)\cdot (x-x_{\bfi}^\ast),\   \  x\in\RR^d,$$
and 
$M_0:=8 m_0^2 b^2 \|\nabla^2 g\|_{L^\infty( [-b,b]^d)}+C_d A_0.$
Here we recall that $A_0$ is the parameter in~\eqref{eqn:a0}. 
Thus, setting 
\begin{align}\label{8-7-1-18}
	S_{\bfi,j}:=\Bl\{ (x,y):  x\in\Delta_{\bfi}^\ast,    H_{\bfi} (x) -\al^\ast_{j+m_1} +\f {M_0} {n^2}\leq y\leq H_{\bfi}(x) -\al^\ast_{j-m_1} -\f {M_0} {n^2}\Br\}
\end{align}
and 
\begin{equation}\label{8-8-1}
	S_{\bfi, j}^\ast:=\Bl\{ (x,y):\  \  x\in\Delta_{\bfi}^\ast,\   \  H_{\bfi} (x) -\al^\ast_{j+m_1} -\f {M_0} {n^2}\leq y\leq H_{\bfi}(x) -\al^\ast_{j-m_1} +\f {M_0} {n^2}\Br\},
\end{equation}
we have  
\begin{equation}\label{8-7-0}
	S_{\bfi,j} \subset I_{\bfi,j}^\ast \subset S_{\bfi,j}^\ast,\   \  (\bfi, j)\in\Ld_n^{d+1}.
\end{equation}
On the other hand, it is easily seen from~\eqref{8-3-0-18},~\eqref{8-5-18} and~\eqref{8-4-18-0} that $S_{\bfi,j}\neq \emptyset$ and 
\begin{equation*} \al^\ast_{j+m_1} -\al^\ast_{j-m_1} -\f {2M_0}{n^2} \sim \f {j+M_0}{n^2}.\end{equation*}
Thus,  $S_{\bfi, j}$ and $S_{\bfi,j}^\ast$ are two nonempty  compact parallelepipeds with the same set $\mathcal{E}(x^\ast_{\bfi})\cup \{e_{d+1}\}$ of edge directions 
and  comparable side lengths.

With the above notations, we introduce the following local modulus of smoothness on $G$:

\begin{defn}\label{def-8-1} For $0<p\leq \infty$, define the local modulus of smoothness of  order $r$  of $f\in L^p( G)$ by  
	$$ \og_{\text{loc}}^r (f, n^{-1})_{L^p( G)}:=
	\Bl[\sum_{(\bfi,j)\in\Ld_n^{d+1}} \Bl(\og^r (f, I_{\bfi,j}^\ast; e_{d+1})_p  ^p+\og^r (f, S_{\bfi,j}; \mathcal{E}(x^\ast_{\bfi}))_p  ^p\Br) \Br]^{1/p},$$
	with the usual change of the   $\ell^p$-norm over the set $(\bfi, j)\in\Ld_n^{d+1}$  for  $p=\infty$.  
\end{defn}

In this section, we shall prove  the following Jackson type estimate for the above  local modulus of smoothness, from which  Theorem~\ref{THM-4-1-18} will follow.

\begin{thm}\label{THM-WT-OMEGA} For $0<p\leq \infty$, and $f\in L^p(G)$,
	\[
	E_{n} (f)_{L^p(G)} \leq C  \omega_{\text{loc}}^r(f,  n^{-1})_{L^p( G)},
	\]
	where the constant  $C$  is independent of $f$ and $n$. 
\end{thm}

\begin{rem}\label{rem:loc mod pnt choice}
	Note that $\omega_{\text{loc}}^r(f,  n^{-1})_{L^p( G)}$ depends on the choice of $x_{\bfi}^\ast$, which is an arbitrary point in $\Delta_{\bfi}^\ast$. It follows from the proof that the constant $C$ in Theorem~\ref{THM-WT-OMEGA} is independent of the selection of the points $x_{\bfi}^\ast\in \Delta_{\bfi}^\ast$.
\end{rem}

We divide the rest of this section  into two parts. In the first part,  we shall assume Theorem~\ref{THM-WT-OMEGA}, and show how it implies Theorem~\ref{THM-4-1-18}, while the second part is devoted to the proof of Theorem~\ref{THM-WT-OMEGA}.  
\subsection{Proof of Theorem~\ref{THM-4-1-18}}\label{subsection-8:1}
The aim  is to show that Theorem~\ref{THM-4-1-18} can be deduced from  Theorem~~\ref{THM-WT-OMEGA}. 
Recall that  for each $\bfi\in\Ld_n^d$,  $\mathcal{E}(x^\ast_{\bfi})$ is the set of  unit  tangent vectors to $\p' G^\ast$ at the point $(x_{\bfi}^\ast, g(x_{\bfi}^\ast))$, where $x_{\bfi}^\ast\in \Delta_{\bfi}^\ast$.   Thus, by Definition~\ref{def-8-1},   Theorem~\ref{THM-WT-OMEGA}, and Remark~\ref{rem:loc mod pnt choice}, to show  Theorem~\ref{THM-4-1-18}, 
it suffices to prove  that for any fixed $c>0$
\begin{equation}\label{8-3-18}
	\Sigma_1:=\sum_{(\bfi,j)\in\Ld_n^{d+1}}  \og^r(f, I_{\bfi, j}^\ast; e_{d+1})^p_p\leq   C\og_{\Og,\vi}^r \Bl(f, \f cn; e_{d+1}\Br)^p_{p},\  
\end{equation}
and for $k=1,\dots, d$,
\begin{equation}\label{8-4-18}
	\Sigma_2 (k):=n^d \sum_{(\bfi, j)\in\Ld_n^{d+1}}  \int_{\Delta_{\bfi}^\ast} \og^r(f, S_{\bfi,j}; \zeta_k(x_{\bfi}^\ast))^p_p dx_{\bfi}^\ast\leq C   \wt{\og}_{G}^r \Bl(f, \f c n\Br)^p_p 
\end{equation}
with the usual change of the $\ell^p$ norm in the case of $p=\infty$.

To prove the estimates~\eqref{8-3-18} and~\eqref{8-4-18}, we need to use  the   average modulus of smoothness  of order $r$ on a compact interval $I=[a_I, b_I]\subset \RR$ defined as   
$$w_r(f, t; I)_p  :=\Bl(\f1t \int_{t/4r}^t\Bl( \int_{I_{rh}} |\tr_h^r f(x)|^p  dx\Br)\, dh  \Br)^{1/p},\     \  0<p\leq \infty,$$
with the usual change when $p=\infty$.
The average modulus $w_r(f, t; I)_p$ turns out to be equivalent to the regular modulus $\og^r(f,t)_p:=\sup_{0<h\leq t} \|\tr_h^r f\|_{L^p(I_{rh})}$,  as is well known.  

\begin{lem}\cite[p.~373, p.~185]{De-Lo} \label{lem-8-1} For $f\in L^p(I)$ and $0<p\leq \infty$,
	\begin{equation}\label{key-equiv-mod-0}
		C_1 w_r(f, t; I)_p \leq \og^r (f,t)_p\leq C_2 w_r(f,t; I)_p,\   \  0<t\leq |I|,
	\end{equation}
	where the constants $C_1, C_2>0$ depend only on $p$ and $r$.
\end{lem}
A consequence of this equivalence and the growth properties of the usual one-dimensional modulus of smoothness (see, e.g. \cite{De-Lo}*{(7.7) and (7.8) on p.~45, (5.8) on p.~370}) is that for $f\in L^p(I)$, $0<p\le\infty$, and any fixed $c\in(0,1)$
\begin{equation}
	\label{eqn:growth-average}
	w_r(f,t;I)_p\le Cw_r(f,ct;I)_p, \quad 0<t\le |I|,
\end{equation}
where $C$ is independent of $f$ and $t$.

For simplicity, we will assume $p<\infty$. The proof below with slight modifications works equally well for the case $p=\infty$. 

We start with the proof of~\eqref{8-3-18}. 
Using~\eqref{8-4-18-0} and~\eqref{key-equiv-mod-0}, we have    
\begin{align*}\og^r(f, I_{\bfi,j}^\ast;  e_{d+1})_p^p &=\sup_{0<h<\f {c_1(j+1)}{n^2}} \int_{\Delta_{\bfi}^\ast}\Bl[ \int_{g(x)-\al_{j+m_1}}^{g(x)-\al_{j-m_1}} |\tr_{h e_{d+1}}^r (f, I_{\bfi, j}^\ast, (x,y))|^p dy\Br] \, dx \\
	&\sim \f {n^2} {j+1}\int_{\f {c_1(j+1)}{4rn^2}}^{\f {c_1(j+1)}{n^2}} \int_{I_{\bfi,j}^\ast } |\tr_{he_{d+1}}^r (f, I_{\bfi,j}^\ast, \xi)|^p d\xi dh.
\end{align*}
By~\eqref{funct-vi}, we note  that for $\xi=(x,y)\in I_{\bfi, j}^\ast-\f {c_1(j+1)}{4n^2} e_{d+1}$, 
\begin{align*} \vi_\Og (e_{d+1}, \xi)&\sim \sqrt{g(x)-y}\sim \f {j+1}n,\   \  0\leq j\leq n.\end{align*}
Thus, 
performing the  change of variable 
$h=s\vi_\Og(e_{d+1}, \xi)$ for each  fixed $\xi\in I_{\bfi,j}^\ast-\f {c_1(j+1)}{4n^2} e_{d+1}$,  we obtain 
\begin{align*}
	& \og^r(f, I_{\bfi,j}^\ast;  e_{d+1})_p^p 
	\leq C   n \int_{I_{\bfi,j}^\ast}\Bl[ \int_{0}^{\f {c_1}n} |\tr_{s\vi_\Og (e_{d+1}, \xi)e_{d+1}}^r(f, I_{\bfi, j}^\ast, \xi)|^p \, ds\Br] d\xi.
\end{align*}
It then follows that 
\begin{align*}
	\Sigma_1
	&\leq Cn\sum_{j=0}^{n-1} \sum_{\bfi\in\Ld_n^d}  \int_0^{\f {c_1}n}\Bl[ \int_{I_{\bfi, j}^\ast} |\tr_{s\vi_\Og (e_{d+1}, \xi)e_{d+1}}^r(f, \Og,x)|^p d\xi\Br] ds\\
	&\leq C n\int_0^{\f {c_1}n} \int_{\Og }|\tr_{u\vi_\Og (e_{d+1}, \xi)e_{d+1}}^r (f,\Og,\xi)|^p\, d\xi ds\leq C \og^r_{\Og,\vi}(f, cn^{-1}; e_{d+1})_p^p,
\end{align*}
where the last step uses~\eqref{eqn:DT-mod-growth}. This proves the estimate~\eqref{8-3-18}.

The estimate~\eqref{8-4-18} can be proved in a similar way.  Indeed, by~\eqref{2-3-18}, \eqref{key-equiv-mod-0} and~\eqref{eqn:growth-average},   it is easily seen  that 
\begin{align*}
	\og^r(f, S_{\bfi, j}; \zeta_k(x_{\bfi}^\ast))_p^p\sim n \int_0^{\f cn} \|\tr_{h \zeta_k(x_{\bfi}^\ast)}^r (f, S_{\bfi, j}) \|_{L^p(S_{\bfi,j})}^p\, dh.
\end{align*}
It follows that 
\begin{align*}
	\Sigma_2(k) &\leq C n^{d+1} \int_0^{\f cn} \Bl[ \sum_{(\bfi, j)\in\Ld_n^{d+1}}  \int_{\Delta_{\bfi}^\ast} 
	\|\tr_{h \zeta_k(x_{\bfi}^\ast)}^r (f, S_{\bfi, j}) \|_{L^p(S_{\bfi,j})}^p\, dx_{\bfi}^\ast\Br]\, dh\\
	& \leq C n^{d}\sup_{0<h\leq \f cn}   \sum_{(\bfi, j)\in\Ld_n^{d+1}}   
	\int_{S_{\bfi, j}} \int_{\|u-\xi_x\|\leq \f c n} |\tr_{h \zeta_k(u)}^r (f, S_{\bfi, j},\xi)|^p\, du\, d\xi\\
	& \leq C  n^{d}\sup_{0<h\leq \f cn}   
	\int_{G^n} \int_{\|u-\xi_x\|\leq \f c n} |\tr_{h \zeta_k(u)}^r (f, G,\xi)|^p\, du\, d\xi\leq C \wt{\og}_G^r(f, \f cn)_p^p,
\end{align*}
where $G^n:=\{\xi\in G:\  \  \dist(\xi, \p' G) \ge \f {A_0}{n^2}\}$. 
This proves~\eqref{8-4-18}.

\subsection{Proof of Theorem~\ref{THM-WT-OMEGA} }\label{subsection-8:2}
The proof relies on several lemmas.

\begin{lem}\label{thm-2-1} Let  $(\bfi,j)\in \Ld_n^{d+1}$. Then for $0<p\leq \infty$, $r\in\NN$ and any $x^\ast_\bfi\in\Delta_{\bfi}^\ast$,  
	$$E_{(d+1)(r-1)}(f)_{L^p(I_{\bfi,j}^\ast)}\leq C(p, r, d, G) \Bl[\og^r (f, I_{\bfi,j}^\ast; e_{d+1})_p+\og^r (f, S_{\bfi,j}; \mathcal{E}(x^\ast_\bfi))_p\Br].$$
\end{lem}

\begin{proof}Lemma~\ref{thm-2-1} follows directly from~\eqref{8-7-0} and Lemma~\ref{cor-7-3}.
\end{proof}

\begin{lem}\label{lem-5-1} Given   $0<p\leq \infty$ and  $r\in\NN$, there  exist positive  constants $C=C(p, r)$ and $s_1=s_1(p,r)$ depending only on $p$ and $r$  such that  for any  integers $0\leq k, j\leq N/2$  and any $P\in\Pi_r^1$,
	\begin{equation*}\label{5-2a}
		\|P\|_{L^p[\al_{j}, \al_{j+1}]}\leq C(p,r) (1+|j-k|)^{s_1}\|P\|_{L^p[\al_{k},\al_{k+1}]}.
	\end{equation*}
\end{lem}
\begin{proof} First, we prove that 
	\begin{equation}\label{5-1}
		\|P\|_{L^p(I_{2t} (x))}  \leq L_{p,r}\|P\|_{L^p(I_t (x))},  \  \   \forall  P\in \Pi_r^1,\  \ \forall x\in [0, 2\al],\  \ \forall t\in (0, 1],
	\end{equation}
	where 
	$$I_t(x):=\Bl\{ y\in [0, 2\al]:\  \ |\sqrt{x}-\sqrt{y}|\leq \sqrt{2\al} t\Br\}.$$
	To see this, we note that   with   $\rho_t(x)=2\al t^2+t\sqrt{2\al  x}$, 
	\begin{equation*}\label{8-17-0} \Bl[x-\f18 \rho_t(x), x+\f18 \rho_t(x)\Br]\cap [0, 2\al]\subset I_t(x)\subset I_{2t}(x) \subset [x-4\rho_t(x), x+4\rho_t(x)],\end{equation*}
	where  the first relation  can be deduced by considering the cases $0\leq x\leq \al t^2$ and $\al t^2<x\leq 4\al$ separately. 
	By Lemma~\ref{lem-4-1}, this implies  that  with $I_t=I_t(x)$ and  $J=[x-4\rho_t(x), x+4\rho_t(x)]$,  
	\begin{align*}
		\|P\|_{L^p(I_{2t})} &\leq \|P\|_{L^p(J)} \leq C_{p,r} \|P\|_{L^p (\f 1{32} J \cap [0, 2\al])} \leq C_{p,r}\|P\|_{L^p(I_t)},
	\end{align*}
	which  proves~\eqref{5-1}.

	Next, we note that  the doubling property~\eqref{5-1} implies that for any $x, x'\in [0, 2\al]$ and any $t\in (0, 1]$, 
	\begin{equation}\label{8-14-18-00}
		\|P\|_{L^p(I_{t} (x))}  \leq L_{p,r} \Bl( 1+ \f {|\sqrt{x}-\sqrt{x'}|}{\sqrt{2\al} t}\Br)^{s_1}\|P\|_{L^p(I_t (x'))},  \  \   \forall  P\in \Pi_r^1,
	\end{equation}
	where $s_1=(\log L_{p,r})/\log 2$.

	Finally, for each $1\leq k\leq N/2$, we may write  $ [\al_k, \al_{k+1}] =I_{t_k} (x_k)$
	with  $t_k:=\f {\sqrt{\al_{k+1}}-\sqrt{\al_k}}{2\sqrt{2\al}}$ and  $x_k :=\f {(\sqrt{\al_k} +\sqrt{\al_{k+1}})^2} 4$. Note also that    
	by~\eqref{8-4-18-0},  
	\begin{align}\label{8-15-18-00}
		\f {\sqrt{2}|k-j|}{2N}\leq 	\f{|\sqrt{\al_j}-\sqrt{ \al_k}|}{\sqrt{2\al}}\leq \f {\pi|k-j|} {2N},\   \ 0\leq k, j\leq N/2.
	\end{align}
	It then follows by~\eqref{8-14-18-00} and~\eqref{8-15-18-00} that 
	\begin{align*}
		\|P\|_{L^p[\al_{j}, \al_{j+1}]}&\leq \|P\|_{L^p (I_{\pi/(4N)}(x_j))} \leq L_{p,r} \Bl ( 1+ \f {4N|\sqrt{x_j}-\sqrt{x_k}|}{\sqrt{2\al} \pi}\Br)^{s_1}
		\|P\|_{L^p (I_{\pi/(4N)}(x_k))} \\
		&\leq L_{p,r}^4 ( 1+|k-j|)^{s_1} \|P\|_{L^p[\al_{k}, \al_{k+1}]}.
	\end{align*}
\end{proof}

For $x=(x_1,\dots, x_d)\in\R^d$, we set  $\|x\|_\infty :=\max_{1\leq j\leq d} |x_j|$. 

\begin{lem}\label{lem-8-4} Given   $0<p\leq \infty$ and  $r\in\NN$, there  exist positive  constants $C=C(p, r, d)$ and $s_2=s_2(p,r,d)$ depending only on $p$, $r$ and $d$ such that  for any  $\ib, \kb\in\Ld_n^d$  and $Q\in\Pi_r^d$,
	\begin{equation*}\label{5-2b}
		\|Q\|_{L^p(\Delta_{\bfi})}\leq C(p,r,d) (1+\|\bfi-\kb\|_\infty)^{s_2}\|Q\|_{L^p(\Delta_{\kb})}.
	\end{equation*}
\end{lem}

\begin{proof} The proof of Lemma~\ref{lem-8-4} is similar to that of Lemma~\ref{lem-5-1}, and in fact, is simpler.  It is a direct consequence of Lemma~\ref{lem-4-1}. 
\end{proof}

\begin{lem}\label{lem-5-2} Given $0<p\leq \infty$ and $r\in\NN$, there exists a positive number $\ell=\ell(p,r,d)$ such that 	
	for any  $(\bfi, j), (\mathbf{k}, l)\in \Ld_n^{d+1}$ and any $Q\in\Pi_r^{d+1}$, 
	\begin{equation}\label{desire}
		\|Q\|_{L^p(I_{\bfi, j})} \leq C \Bl(1+\max\{\|\bfi-\mathbf{k}\|_\infty,  |j-l|\}\Br)^{\ell}\|Q\|_{L^p(I_{\mathbf{k}, l})},
	\end{equation}
	where the constant $C$ depends only on $p, d, r$ and $\|\nabla^2 g\|_\infty$.
\end{lem}
\begin{proof} For simplicity, we shall prove Lemma~\ref{lem-5-2} for the case of $0<p<\infty$ only. The proof below with slight modifications works for  $p=\infty$.

	Writing 
	$$\|Q\|^p_{L^p(I_{\bfi, j})} = \int_{\Delta_{\bfi}}\Bl[\int_{\al_{j-1}}
	^{\al_{j}} |Q({x}, g({x})-u)|^p\, du\Br] d{x},$$
	and using  Lemma~\ref{lem-5-1}, we  obtain 
	\begin{align}\label{8-18-0}
		\|Q\|^p_{L^p(I_{\bfi, j})}\leq C(p,r) (1+|j-l|)^{s_1} \Bl[ \int_{\Delta_{\bfi}}\int_{g({x})-\al_{l}}
		^{g({x})-\al_{l-1}} |Q({x}, y)|^p\, dy d{x}\Br].
	\end{align}
	Using   Taylor's theorem, we have that
	\begin{equation}\label{5-3}
		|g({x})-t_{\bfi} ({x})|\leq \f {A}{2b^2}\|{x}-{x}_{\bfi}\|_\infty^2,\    \    \ \forall x\in [-b,b]^d,
	\end{equation}
	where   ${x}_{\bf i}$ is the center of the cube $\Delta_{\bfi}$, 
	$t_{\bfi}({x}):=g({x}_{\bfi})+\nabla g({x}_{\bfi})\cdot ({x}-{x}_{\bfi})$, and 
	$A:=2b^2d^2\|\nabla^2 g\|_{L^\infty[-b,b]^d}/2$.
	Thus,  the double  integral in the square brackets on the right hand side of~\eqref{8-18-0} is bounded above by 
	\begin{align*}
		&   \int_{\al_{\ell-1}}^{\al_\ell+\f {A}{n^2}}\Bl[  \int_{\Delta_{\bfi}}\Bl|Q\Bl({x}, t_{\bfi}({x})
		-u+\f A {2n^2}\Br)\Br|^p\, d{x}\Br] du=: I.
	\end{align*}
	However, applying  Lemma~\ref{lem-8-4} to this last  inner integral in the square brackets, we obtain 
	\begin{align}\label{8-20-1}
		I&\leq C(p,r,d)     (1+\|\bfi-\mathbf{k}\|_\infty)^{s_2} \int_{\Delta_{\mathbf k}}\Bl[\int_{t_{\bfi}({x})
			-\al_\ell-\f A {2n^2}}^{t_{\bfi}({x})
			-\al_{\ell-1}+\f A {2n^2}} |Q({x}, u)|^p\, du\Br] d {x}.
	\end{align}
	By~\eqref{5-3},  this last   integral in the square brackets on the right hand side of~\eqref{8-20-1} is bounded above by 
	\begin{align*}
		& \int_{g({x})
			-\al_\ell-\f {4A(1+\|\mathbf{k}-\bfi\|_\infty^2)}{n^2}}^{g({x})
			-\al_{\ell-1}+\f {4A(1+\|\mathbf{k}-\bfi\|_\infty^2)}{n^2}} |Q({x}, u)|^p\, du 
		=\int^{
			\al_\ell+\f {4A(1+\|\mathbf{k}-\bfi\|_\infty^2)}{n^2}}_{\al_{\ell-1}-\f {4A(1+\|\mathbf{k}-\bfi\|_\infty^2)}{n^2}} |Q({x}, g({x})-y)|^p\, dy, \end{align*}
	which,  using  Lemma~\ref{lem-4-1} and the fact that $\al_\ell-\al_{\ell-1}\ge c n^{-2}$,  is controlled above  by 
	\begin{align*}
		C(p,r)  \Bl(A (1+\|\mathbf {k}-\bfi\|_\infty)\Br)^{2rp+4}  \int^{\al_\ell}_{
			\al_{\ell-1}} |Q({x}, g({x})-y)|^p\, dy.
	\end{align*}

	Putting the above together, we prove  that 
	\begin{align*}
		\|Q\|^p_{L^p(I_{\bfi, j})}\leq C(p,r,d) (1+|j-l|)^{s_1}   (1+\|\bfi-\mathbf{k}\|_\infty)^{s_2+2rp+4}\|Q\|^p_{L^p(I_{\kb, l})}.
	\end{align*}
	This leads to the desired estimate~\eqref{desire} with $\ell= (s_1+s_2+2rp+4)/p$.
\end{proof}

Now we are in the position to prove Theorem~\ref{THM-WT-OMEGA}. 
\begin{proof}[Proof of Theorem~\ref{THM-WT-OMEGA}]  We shall prove the result for the case of   $0<p<\infty$ only.
	The proof below with slight modifications works equally well for the case $p=\infty$.
	
	For simplicity, we  use the Greek letters $\ga, \b,\dots$ to denote  indices in the set $\Ld_n^{d+1}$. 
	By Lemma~\ref{thm-2-1}, for each  $\ga:=(\ib, j)\in\Ld_n^{d+1}$  there exists a polynomial  $s_{\ga}\in\Pi_{(d+1)(r-1)}^{d+1}$  such that 
	\begin{align}\label{5-4}
		\|f-s_{\ga}\|_{L^p(I_{\ga}^\ast)} \leq C(p, r, d) W^r(f, I_{\ga}^\ast)_p,
	\end{align}
	where 
	$$  W^r(f, I_{\ga}^\ast)_p:=\og^r (f, I_{\ga}^\ast; e_{d+1})_p+\og^r (f, S_{\ga}; \mathcal{E}(x^\ast_\bfi))_p.$$
	Let $\{q_{\ga}:\  \ \ga\in\Ld_n^{d+1}\}\subset \Pi_{\lfloor n/(r(d+1))\rfloor}^{d+1}$ be the polynomial  partition of the unity as  given in Theorem~\ref{strips-0} and  Remark~\ref{rem-6-3} with a  large parameter $m>2d+2$, to be specified later. Define 
	$$P_n(\xi):=\sum_{\ga\in\Ld_n^{d+1}} s_\ga (\xi) q_\ga(\xi)\in\Pi_{n}^{d+1}.$$
	Clearly, it is sufficient  to prove that 
	\begin{equation}\label{8-22-00}
		\|f-P_n\|_{L^p(G)}\leq C \og^r_{\text{loc}} \Bl(f, \f1n\Br)_p.
	\end{equation}

	To show~\eqref{8-22-00},  we write, for each  $\b\in \Ld_n^{d+1}$,
	\begin{align*}
		f(\xi)-P_n(\xi)&=f(\xi)-s_\b(\xi)+\sum_{\ga\in \Ld_n^{d+1}} (s_\b(\xi)-s_\ga (\xi))q_\ga (\xi).
	\end{align*}
	It follows by Theorem~\ref{strips-0}  that 
	\begin{align*}
		\|f-P_n\|_{L^p(I_\b)}^p &\leq C_p \|f-s_\b\|_{L^p(I_\b)}^p +C_p \sum_{\ga\in\Ld_n^{d+1}} \|s_\b-s_\ga\|^p_{L^p(I_\b)} (1+\|\b-\ga\|_\infty)^{-mp_1},
	\end{align*}
	where $p_1:=\min\{p,1\}$.
	Using~\eqref{5-4}, we then  reduce  to showing that   \begin{align}\label{8-23-00}
		\Sigma_n'&:=\sum_{\b\in\Ld_n^{d+1}}\sum_{\ga\in\Ld_n^{d+1}} \|s_\b-s_\ga\|^p_{L^p(I_\b)} (1+\|\b-\ga\|_\infty)^{-mp_1}\leq C \og_{\text{loc}}^r\Bl(f,\f1n\Br)_p^p.
	\end{align}
	
	To show~\eqref{8-23-00},  we  claim that  there exists a positive number $s_3=s_3(p,d,r)$ such that for any  $\ga,\b\in\Ld_n^{d+1}$,
	\begin{equation}\label{claim-8-23} \|s_\ga-s_\b\|^p_{L^p(I_\ga)} \leq C (1+\|\ga-\b\|_\infty)^{s_3 p} \sum_{\eta\in \mathcal{I}_{k_0} ( \ga)} W^r(f, I^\ast_\eta)^p_p,\end{equation}
	where  $k_0:=1+\|\b-\ga\|_\infty$, and
	$$ \mathcal{I}_t(\ga):=\{ \eta\in\Ld_n^{d+1}:\  \ \|\ga-\eta\|_\infty\leq t\}\   \ \text{for $ \ga\in\Ld_n^{d+1}$ and $t>0$}.$$
	For the moment, we assume~\eqref{claim-8-23} and proceed with the proof of~\eqref{8-23-00}. Indeed, 
	we have 
	\begin{align*}
		\Sigma_n'  &\leq C\sum_{\b\in\Ld_n^{d+1}}\sum_{k=1}^\infty  k^{-mp_1}\sum_{\ga\in \mathcal{I}_k(\b)\setminus \mathcal{I}_{k-1}(\b)}\|s_\b-s_\ga\|_{L^p(I_\b)}^p,\end{align*}
	which,      using~\eqref{claim-8-23}, is bounded above by 
	\begin{align*}
		&  C\sum_{k=1}^\infty k^{-mp_1+s_3p+2d+2} \sum_{\eta\in  \Ld_n^{d+1}}W^r(f, I_\eta^\ast)_p^p.
	\end{align*}  Choosing  the parameter $m$ to be bigger than $s_3p/p_1 + (2d+4)/p_1$, we then  prove~\eqref{8-23-00}.

	It remains to prove the claim~\eqref{claim-8-23}.  A crucial ingredient in the proof   is to construct a sequence $\{\ga_1,\dots, \ga_{N_0}\}$ of distinct indices in $\Ld_n^{d+1}$ with the properties that  $N_0\leq  C ( 1+\|\ga-\b\|_\infty)^2$,
	$\ga_1=\ga$, $\ga_{N_0}=\b$, and for $j=0,\dots, N_0-1$,
	\begin{align}\label{8-25-00}
		I_{\ga_j} \subset I_{\ga_{j+1}}^\ast\   \  \text{and}\   \   \|\ga_j-\ga\|\leq 1+\|\ga-\b\|.
	\end{align}
	Indeed, once such a sequence is constructed, then 
	we have 
	\begin{align*}
		\|s_\ga-s_\b\|_{L^p(I_\ga)}^p&\leq N_0^{\max\{p,1\}-1}  \sum_{j=1}^{N_0-1}\|s_{\ga_j}-s_{\ga_{j+1}}\|^p_{L^p(I_{\ga})}, \end{align*}
	which, using~\eqref{8-25-00} and  Lemma~\ref{lem-5-2} with $\ell=\ell(p,r,d)>0$, is estimated above by
	\begin{align*}
		&\leq  C N_0^{\max\{p,1\}-1} (1+\|\ga-\b\|_\infty)^{\ell p}  \sum_{j=1}^{N_0-1}\|s_{\ga_j}-s_{\ga_{j+1}}\|^p_{L^p(I_{\ga_j})}.\end{align*}
	However, using~\eqref{8-25-00} and~\eqref{5-4}, we have that 
	\begin{align*}
		\|s_{\ga_j}-s_{\ga_{j+1}}\|^p_{L^p(I_{\ga_j})}\leq &C_p \Bl[ \|f-s_{\ga_j}\|_{L^p(I_{\ga_j})}^p +\|f-s_{\ga_{j+1}}\|_{L^p(I^\ast_{\ga_{j+1}})}^p\Br]\\
		\leq& C(p,r,d)\Bl[  W^r(f, I_{\ga_j}^\ast)_p^p+ W^r(f, I_{\ga_{j+1}}^\ast)_p^p\Br].
	\end{align*}
	Putting the above together, we prove the claim~\eqref{claim-8-23} with $s_3:=\ell+2\max\{1, \f 1p\}$.

	Finally, we construct the sequence $\{\ga_1,\dots, \ga_{N_0}\}$ as follows.  Assume that  $\ga=(\mathbf{k}, l)$, and  $\b=(\mathbf{k}', l')$. Without loss of generality, we may assume that   $l\leq l'$. (The case $l>l'$ can be treated similarly.)
	Recall that 
	$ \Delta_{\ib}:=\Bl\{x\in \RR^d:\  \ \|x-x_{\ib}\|_\infty\leq \f b n\Br\},$
	where  ${x}_{\bfi}$ is  the center of the cube $\Delta_{\bfi}$.
	Let $\{z_j\}_{j=0}^{n_0+1}$ be a sequence of points on the line segment $[x_{\kb}, x_{\kb'}]$ satisfying that  $z_0 =x_{\kb}$, $z_{n_0+1} =x_{\kb'}$, $\|z_j-z_{j+1}\|_\infty =\f {3b} n $ for $j=0,1,\dots, n_0-1$  and $\f {3b} n \leq \|z_{n_0}-z_{n_0+1}\|_\infty<\f {6b}n$, where $n_0+1\leq  \f 23 \|\kb-\kb'\|_\infty$.  Let $\ib_j \in\Ld_n^d$ be such that $z_j\in\Delta_{\ib_j}$ for $0\leq j\leq n_0+1$. 
	Since $\f {3b} n \leq \|z_j-z_{j+1}\|_\infty \leq  \f {6b} n$, the cubes $\Delta_{\ib_j}$ are distinct and  moreover
	\begin{equation}\label{8-25-0}
		\Delta_{\ib_j} \subset 9 \Delta_{\ib_{j+1}},\  \  j=0,1,\dots, n_0.
	\end{equation}
	In particular, this implies that $\ib_0=\kb$ and $\ib_{n_0+1} =\kb'$.
	It can also be easily seen from the construction  that for $j=0,\dots, n_0+1$, 
	\begin{equation}\label{8-27}
		\|\ib_j -\kb\|_\infty \leq \|\kb-\kb'\|_\infty+1.
	\end{equation}
	Next,  we order   the indices $(\ib_j, k)$, $0\leq j\leq n_0+1$, $l\leq k\leq l'$ as follows: 
	\begin{align*}
		(\ib_0, l), (\ib_0, l+1),\dots, (\ib_0, l'),
		(\ib_1, l'), (\ib_1, l'-1),\dots, (\ib_1, l), (\ib_2, l),\dots, (\ib_{n_0+1}, l').      
	\end{align*}
	We  denote the resulting  sequence by  
	$ \{\ga_1, \ga_2,\dots, \ga_{N_0}\},$
	where
	$$N_0\leq (1+|l-l'|) (n_0+2) \leq ( 1+\|\ga-\b\|_\infty)^2.$$ 
	Clearly, $\ga_1=\ga$, and  $\ga_{N_0}=\b$. Moreover, by~\eqref{8-27}, we have $\|\ga_j-\ga\|_\infty\leq 1+\|\ga-\b\|_\infty$ for $j=1,\dots, N_0$, whereas by~\eqref{8-25-0}, 
	$I_{\ga_j} \subset I_{\ga_{j+1}}^\ast $ for  $j=0,\dots, N_0-1$.
	This completes the proof. 
\end{proof}

\section{Comparison with average  moduli}\label{ch:IvanovModuli}
In this section, we shall prove that the   moduli of smoothness,   defined in~\eqref{eqn:defmodulus}  can be controlled from above by  Ivanov's moduli of smoothness, defined  in~\eqref{eqn:ivanov}.  By Remark~\ref{rem-3-2}, it is enough to show 

\begin{thm}\label{thm-9-1}
	There exist 	a parameter $A_0>1$  and a constant $A>1$ such that for any $0<q\leq p\leq \infty$,
	$$\og_\Og^r\Bl(f, \f 1n; A_0\Br)_p \leq C \tau_r \Bl(f, \f A n\Br)_{p,q},$$
	where the constant $C$ is independent of $f$ and $n$. 
\end{thm}

As a result, using Remark~\ref{rem:constant-in-jackson}, we may establish the Jackson inequality for Ivanov's moduli of smoothness for  any dimension   $d\ge 1$ and  the full range of $0<q\leq p\leq \infty$.

\begin{cor}  	
	If $f\in L^p(\Og)$, $0< q\leq p \leq \infty$ and $r\in\NN$, then
	$$ E_n (f)_p \leq C_{r, \Og} \tau_r \Bl(f, \frac 1 n\Br)_{p,q}.$$
\end{cor}

Recall that for  $S\subset \R^d$,
$$S_{rh}:=\Bl\{\xi\in S:\  \ [\xi, \xi+rh]\subset S\Br\},   \   r>0, h\in\R^d.$$
The proof of Theorem~\ref{thm-9-1}  relies on the following lemma, which  generalizes Lemma~7.4 of~\cite{Di-Pr08}.
\begin{lem}\label{lem-9-1:Dec}
	Let     $r\in\NN$,   $h\in\R^d$ and $\da_0\in (0,1)$.  Assume that  $(S,E)$ is a pair of  subsets   of $\R^d$  satisfying  that  for  each $\xi\in S_{rh}$,  there exists   a convex subset $E^\xi$ of $E$ such that  $|E^\xi|\ge \da_0 |E|$
	and   $[\xi, \xi+rh]\subset E^\xi$.
	Then for any $0<q\leq p <\infty$ and $f\in L^p(E)$, we have 
	\begin{equation*}\label{9-1-18}
		\|\tr_h^r (f, S, \cdot)\|_{L^p(S)} \leq  C(q, d, r)\Bl(\int_S \Bl(\f 1 {\da_0 |E|}  \int_E \bl| \tr_{(\eta-\xi)/r} ^r (f, E, \xi)\br|^q\, d\xi\Br)^{\f pq}\, d\eta\Br)^{\f1p},
	\end{equation*}
	where the constant $C(q,d, r)$ is independent of $S$, $E$ and $q$ if $q\ge 1$.
\end{lem}
Lemma~\ref{lem-9-1:Dec} was  proved in \cite[Lemma 7.4]{Di-Pr08}
in the case  when  $p=q$ and   $E=S$ is convex.   For the general case, it can be obtained by  modifying the proof there.  
\begin{proof}   
	The proof is based  on the following combinatorial identity, which was  proved in \cite[Lemma 7.3]{Di-Pr08}:   if $\xi, \eta\in\R^d$ and $f$ is defined on the convex hull of the set $\{\xi, \xi+rh, \eta\}$, then 
	\begin{align}\label{9-2-0}
		\tr^r_h f(\xi) =&\sum_{j=0}^{r-1} (-1)^j \binom r j \tr^r f\Bl[\xi+jh, \  \f jr (\xi+rh)+\Bl(1-\f jr\Br)\eta\Br]\\
		&	-\sum_{j=1}^r (-1)^j \binom r j \tr^r f \Bl[ \Bl(1-\f jr\Br) \xi+\f jr \eta,\  \  \xi+rh \Br],\notag
	\end{align}
	where  we used the notation   
	$\tr^r f[u, v] :=\tr_{(v-u)/r} ^r f(u)$ for $u,v\in\RR^d$. 
	
	Since $E^\xi$ is a convex set containing the line segment $[\xi, \xi+rh]$ for each $\xi\in S_{rh}$, we obtain from~\eqref{9-2-0} that   for  $\xi\in S_{rh}$,
	\begin{align*}
		|\tr_h^r f(\xi) |\leq&  C_r \max_{0\leq j\leq r-1}
		\Bl(\f 1 {|E^\xi|}	\int_{E^\xi}	\Bl|\tr^rf\bl[\xi+jh, \  \f jr (\xi+rh)+\Bl(1-\f jr\Br)\eta\br]\Br|^q\, d\eta\Br)^{\f1q}\\
		&+  C_r \max_{1\leq j\leq r}
		\Bl(\f 1 {|E^\xi|}	\int_{E^\xi}\Bl|\tr^r f \bl[ \Bl(1-\f jr\Br) \xi+\f jr \eta,\  \  \xi+rh \br]\Br|^q\, d\eta\Br)^{\f1q}.
	\end{align*}
	Taking the $L^p$-norm  over the set  $S_{rh}$ on both sides of this last inequality,   we obtain
	\begin{align*}
		\Bl(\int_{S_{rh}}	|\tr_h^r f(\xi) |^p\, d\xi\Br)^{\f1p}\leq&  C_{r}\da_0^{-\f1q} \Bl[\max_{0\leq j\leq r-1} I_j(h) + \max_{1\leq j\leq r}K_j(h)\Br],
	\end{align*}
	where 
	\begin{align*}
		I_j(h):&=\Bl(\int_{S_{rh}}\Bl(\f 1{|E|}	\int_{E^\xi}	\Bl|\tr^r f\bl[\xi+jh, \  \f jr (\xi+rh)+\Bl(1-\f jr\Br)\eta\br]\Br|^q\, d\eta\Br)^{\f pq}\, d\xi\Br)^{\f1p},\\
		K_j(h):&=\Bl(\int_{S_{rh}}\Bl(\f 1{|E|}	\int_{E^\xi}\Bl|\tr^r f \bl[ \Bl(1-\f jr\Br) \xi+\f jr \eta,\  \  \xi+rh \br]\Br|^q\, d\eta\Br)^{\f pq}\, d\xi\Br)^{\f1p}.	
	\end{align*}
	For the term   $I_j(h)$ with  $0\leq j\leq r-1$, we have
	\begin{align*}
		I_j(h)&=\Bl(\int_{S_{rh}+jh}\Bl(\f 1 {|E|}	\int_{E^{u-jh}}	\Bl|\tr^r f\bl[u, \  \f jr (u+(r-j)h)+\Bl(1-\f jr\Br)\eta\br]\Br|^q\, d\eta\Br)^{\f pq}\, du\Br)^{\f1p}\\
		&\leq r^{d/q}
		\Bl(\int_{S_{rh}+jh}\Bl(\f 1 {|E|}	\int_{E^{u-jh}}	\Bl|\tr^rf[u, \  v]\Br|^q\, dv\Br)^{\f pq}\, du\Br)^{\f1p},
	\end{align*}
	where we used the change of variables  $u=\xi+jh$ in the first step,  the change of variables  $v=\f jr (u+(r-j) h)+\Bl(1-\f jr\Br) \eta$  and the fact that each set  $E^\xi$ is convex  in the second step.
	Since 
	$[u,v] \subset E^{u-jh} \subset E$  whenever $u\in S_{rh}+jh$ and $v\in E^{u-jh}$ and since  $\tr^r f [u,v]=\tr^r f[v,u]$, it follows that 
	$$ I_j(h) \leq r^{d/q} 	\Bl(\int_{S}\Bl(\f 1 {|E|}	\int_{E}	\Bl|\tr^r_{(u-v)/r} (f, E, v)\Br|^q\, dv\Br)^{\f pq}\, du\Br)^{\f1p}.$$
	
	The  terms $K_j(h)$,  $1\leq j\leq r$ can be estimated in a similar way. In fact,  making  the change of variables $u=\xi+rh$ and $v=\Bl(1-\f jr\Br)(u-rh) +\f jr \eta$, we   obtain 
	\begin{align*}
		K_j(h)&=\Bl(\int_{S_{-rh}} 
		\Bl(\f 1 {|E|} 	\int_{E^{u-rh}}\Bl|\tr^r f \bl[ \Bl(1-\f jr\Br) (u-rh)+\f jr \eta,\  \  u \br]\Br|^q\, d\eta\Br)^{\f pq}\, du\Br)^{\f1p}\\
		&\leq r^{d/q} \Bl(\int_{S_{-rh}} 
		\Bl(\f 1 {|E|} 	\int_{E^{u-rh}}\Bl|\tr^r f [v,\  \  u ]\Br|^q\, dv\Br)^{\f pq}\, du\Br)^{\f1p}\\
		&\leq r^{d/q} 	\Bl(\int_{S}\Bl(\f 1 {|E|}	\int_{E}	\Bl|\tr^r_{(u-v)/r} (f, E, v)\Br|^q\, dv\Br)^{\f pq}\, du\Br)^{\f1p}.
	\end{align*}
	
	Putting the above together,  we complete  the proof. 	
\end{proof}

We are now in a position to prove Theorem~\ref{thm-9-1}.

\begin{proof}[Proof of  Theorem~\ref{thm-9-1}]
	We shall  prove Theorem~\ref{thm-9-1} for  $p<\infty$ only.
	The case $p=\infty$ can be deduced by letting $p\to\infty$. In fact,  all the general constants below are independent of $p$ as $p\to\infty$.

	
	
	By Lemma~\ref{lem-2-1-18}, there exists $\da_0\in (0,1)$ such that 
	$\Og\setminus \Og(\da_0) \subset \bigcup_{j=1}^{m_0} G_j,$
	where 
	$$\Og(\da_0):=\{\xi\in\Og:\  \  \dist(\xi, \Ga) > \da_0\}.$$
	We claim that 
	for any $ 0<t< \f {\da_0}{ 8\diam (\Og) +8}$,
	\begin{align}\label{9-5-2}
		\sup_{\|h\|\leq  t} \Bl\|\tr_{h\vi_{\Og} (h, \cdot)}^r (f, \Og, \cdot)\Br\|_{L^p(\Og(\da_0))}\leq C_{q,d} \tau_r (f, A_1t)_{p,q}.
	\end{align} Indeed,  using Fubini's theorem and Lemma~\ref{lem-8-1},  we have 
	\begin{align*}
		\sup_{\|h\|\leq  t} \Bl\|\tr_{h\vi_{\Og} (h, \cdot)}^r f\Br\|_{L^p(\Og(\da_0))}\leq C_d  \sup_{\|h\|\leq t} \Bl\|\tr_{h}^r f\Br\|_{L^p(\Og(\da_0/2))}.
	\end{align*}
	Let $\{\og_1,\dots, \og_{N}\}$ be a subset of $\Og(\da_0/2)$ such that $\min_{1\leq i\neq j\leq N} \|\og_i-\og_j\|\ge t$ and 
	$\Og(\da_0/2) \subset \bigcup_{j=1}^{N} B_j$,  where $B_j:=B_{t} (\og_j)$.
	Using Lemma~~\ref{thm-9-1}, we then  have 
	\begin{align*}
		&\sup_{\|h\|\leq t} \Bl\|\tr_{h}^r f\Br\|^p_{L^p(\Og(\da_0/2))}\leq C_{p} \sum_{j=1}^{N} \sup_{\|h\|\leq t} \Bl\|\tr_{h}^r(f, 2B_j, \cdot)\Br\|^p_{L^p(B_j)}\\
		&\leq C_{q}  \sum_{j=1}^{N}\int_{2B_j} \Bl(\f 1{t^{d+1}} \int_{ B_{4t} (\xi)} |\tr_{(\eta-\xi)/r}^r f (\xi)|^q \, d\xi\Br)^{\f pq} \, d\eta\leq C_{q,d} \tau_r (f, A_1 t)_{p,q}^p.
	\end{align*}
	This proves the claim~\eqref{9-5-2}.

	Now using~\eqref{9-5-2} and Definition~\ref{def:modulus}, we   reduce to showing that  for each $x_i$-domain $G\subset \Og$   attached to $\Ga$, and a sufficiently large parameter $A_0$, 
	\begin{equation}\label{9-3}
		\wt{\og}^r_G (f, \f1n; A_0)_{L^p(G)}\leq C  \tau_r \Bl(f, \f {A_1} n\Br)_{p,q}
	\end{equation}
	and  
	\begin{equation}\label{9-4}
		\sup_{0<s\leq  \f 1n} \|\tr_{s \vi_{\Og} (e_i,\cdot)e_i }^r (f, G,\cdot)\|_{L^p(G)}\leq C  \tau_r \Bl(f, \f {A_1} n\Br)_{p,q}.
	\end{equation}

	Without loss of generality, we may assume that $e_i=e_{d+1}$, $G$ takes the form~\eqref{standard} with small base size  $b\in(0,1)$, and $n\ge N_0$, where $N_0$ is a large positive integer depending only on the set $\Og$.
	We  follow  the same notations as   in  Section~\ref{Sec:8} with sufficiently large  parameters $m_0$ and $m_1$. Thus,  	
	$\{I_{\ib, j}:\  \  (\ib, j) \in\Ld_{n}^{d+1}\}$ is a partition of $G$, and  $S_{\ib,j}\subset I_{\ib,j}$ is  the compact parallelepiped as defined in~\eqref{8-7-1-18}.

	We start with the proof of~\eqref{9-3}.
	Given a  parameter $\ell>1$, we   define 
	\begin{align*}
		S_{\bfi,j}^ \diamond:=\Bl\{ (x,y):\  \  & x\in(\ell \Delta_{\bfi}^\ast)\cap [-2b, 2b]^d,\   \  H_{\bfi} (x) -\al^\ast_{j+m_1} +\f {M_0-\ell} {n^2}\leq y\leq\\
		&\leq  H_{\bfi}(x) -\al^\ast_{j-m_1} -\f {M_0-\ell} {n^2}\Br\},
	\end{align*}
	where $\ell \Delta_{\ib}^\ast$ denotes the dilation of the cube  $\Delta_{\ib}^\ast$ from its center $x_{\ib}$.  
	We  choose   the parameter $\ell$ sufficiently large  so that 
	\begin{enumerate}[\rm (i)]
		\item for any $\xi=(\xi_x,\xi_y)\in I_{\ib,j}$ and  $u\in B_{n^{-1} } (\xi_x)\subset \R^d$,  	$ \Bl[\xi, \xi +\f rn \zeta_k(u)\Br]\subset      S_{\ib,j}^ \diamond$ for all $1\leq k\leq d$; 
		\item  there exists a constant $c_0>0$ such that 
		$I_{\ib, j} \subset S_{\ib,j}^{\diamond}\subset G^\ast$ whenever  $\ib\in\Ld_n^{d}$ and  $j\ge c_0 \ell$.
	\end{enumerate}
	Furthermore, we may also choose the parameter $A_0$  large enough  so that 
	with $\Ld_{n,\ell}^{d+1}:=\{(\ib,j)\in \Ld_n^{d+1}:\   \  c_0\ell\leq j\leq n\}$, 
	$$   G_n:=\Bl\{\xi\in G:\  \  \dist(\xi, \p' G) \ge  \f {A_0} {n^2}\Br\}\subset \bigcup_{(\ib,j)\in\Ld_{n,\ell}^{d+1}} I_{\ib,j}.$$
	With the above notation, we  have that  for any $0<s\leq \f 1n$ and $k=1,\dots, d$, 	
	\begin{align*}
		&n^d\int_{G_n} \int_{\|u-\xi_x\|\leq \f1n} |\tr_{s \zeta_k(u)}^r (f, G^\ast,\xi)|^p  \, du d\xi\leq C_d \sum_{(\ib,j)\in\Ld_{n,\ell}^{d+1}} \sup_{\zeta \in\SS^d} \int_{S^{\diamond}_{\ib,j}}  |\tr_{s \zeta}^r (f, S_{\ib,j}^\diamond,\xi)|^p  d\xi,\end{align*}
	which, using  Lemma~\ref{lem-9-1:Dec}, is estimated above  by 
	\begin{equation}\label{9-5}
		C_{q,d,r} \sum_{(\ib,j)\in\Ld_{n,\ell}^{d+1}}
		\int_{S_{\ib,j}^\diamond}\Bl( \f 1 {|S_{\ib,j}^{\diamond}|}  \int_{S_{\ib,j}^\diamond} \bl| \tr_{(\eta-\xi)/r} ^r f( \xi)\br|^q\, d\xi\Br)^{\f pq}\, d\eta.
	\end{equation}
	Recall that for $\xi\in\Og$ and $t>0$, we defined 
	$ U( \xi, t)= \{\eta\in\Og:\  \  \rho_\Og(\xi,\eta) \leq t\}$. Now, by Proposition~\ref{metric-lem},  there exists a constant $A_1>1$ such that for each $(\ib, j)\in\Ld_{n,\ell}^{d+1}$, 
	$$    U\Bl(\eta_{\ib,j}, \f 1{n A_1}\Br)\subset  S_{\ib, j}^{\diamond}\subset  U\Bl(\eta_{\ib,j}, \f {A_1}{2n}\Br) \    \  \text{for some  $\eta_{\ib,j}\in S_{\ib,j}^\diamond$}.$$
	Thus,  by  Remark~\ref{rem-6-2},  the sum in~\eqref{9-5} is controlled above by  a constant multiple of 
	\begin{align*}
		\int_{\Og}\Bl( \f 1 {|U(\xi, \f {A_1} n)|}  \int_{U(\xi, \f {A_1} n)} \bl| \tr_{(\eta-\xi)/r} ^r (f,\Og, \xi)\br|^q\, d\xi\Br)^{\f pq}\, d\eta= \tau_r \Bl(f, \f {A_1} n\Br)_{p,q}^p.
	\end{align*}
	This completes the proof of~\eqref{9-3}.

	It remains to prove~\eqref{9-4}.  First, by the $C^2$ assumption of the domain $\Og$ (see, e.g.~\cite{Wa}), there exists a constant $r_0\in (0,1)$ such that for each  $\xi=(\xi_x, \xi_y)\in G$,  there exists a closed ball   $B_\xi\subset G^\ast$  of  radius $r_0\in (0, 1)$ that touches the boundary  $\Ga$ at the point $\ga(\xi):=(\xi_x, g(\xi_x))$. 
	Given   a large parameter  $A$, we 	
	define
	\begin{equation}\label{9-6}
		E_\xi:=\Bl\{ \eta \in  B_\xi:\  \  \dist(\eta, T_\xi)\leq \f {A} {n^2}\Br\},\   \  \xi\in G,
	\end{equation}
	where  $T_\xi$ denotes  the tangent plane to $\Ga$ at the point $\ga(\xi)$. Clearly,   $E_\xi\subset G^\ast$ is convex, 
	\begin{equation}\label{9-6-0}
		U(\ga(\xi), \f {c_1} {n}) \subset E_\xi\subset U(\ga(\xi), \f {c_2}n),
	\end{equation}
	where the constants  $c_1, c_2>0$ depend only on $G$ and the parameter $A$. 
	Next, recall that  $S_{\ib,j}^\ast$ is  the compact parallelepiped  defined in~\eqref{8-8-1}. By definition, there exists a positive integer $j_0$ depending only on $G$ such that $S_{\ib,j}^\ast\subset G^\ast$ whenever $j_0<j\leq n$.  Furthermore,    according to  Proposition~\ref{metric-lem}, we have that
	\begin{equation}\label{9-8-1}
		\sup_{\xi\in S^\ast_{\ib, j}} \|\xi-\ga(\xi)\|\leq \f {c_3} {n^2},\   \ \text{ for $0\leq j\leq j_0$, }
	\end{equation}
	and    
	\begin{equation}\label{9-7}U\Bl(\eta_{\ib,j}, \f {c_4} n\Br) \subset I^\ast_{\ib,j} \subset S_{\ib,j}^\ast\cap G^\ast  \subset U\Bl(\eta_{\ib,j}, \f {c_5} n\Br),\    \  \forall (\ib,j) \in\Ld_n^{d+1},\end{equation}
	for some point   $\eta_{\ib,j} \in I_{\ib,j}$,  
	where $c_3, c_4, c_5$ are positive constants depending only on the set $G$. 
	By~\eqref{9-8-1}, we may choose the parameter $A$ in~\eqref{9-6} large enough so that if $0\leq j\leq j_0$ and  $\xi\in I^\ast_{\ib,j}$, then  $[\xi, \ga(\xi)]\subset E_\xi$. 
	Note that if $\xi\in I_{\ib,j}^\ast$ with  $0\leq j\leq j_0$, then by~\eqref{9-7} and~\eqref{9-8-1}, 
	$$\rho_{\Og} (\eta_{\ib,j}, \ga(\xi)) \leq \f {c_6}n,$$
	where  $c_6>0$ is a constant depending only on $G$.    
	Now we define, for   $(\ib,j) \in\Ld_n^{d+1}$, 
	$$ E_{\ib, j} =\begin{cases}
		S_{\ib,j}^\ast, \   \  & \text{ if $j_0<j\leq n$}, \\
		U(\eta_{\ib,j}, \f {c_2+c_6} n), \  \ & \text{ if $0\leq j\leq j_0$}.
	\end{cases}$$
	Thus,  $E_{\ib, j}\subset G^\ast$, and  by~\eqref{9-6-0},~\eqref{9-8-1} and~\eqref{9-7},  we have  that for $\leq j\leq j_0$.
	\begin{equation}\label{9-10}
		\bigcup_{\xi\in I^\ast_{\ib,j}} E_\xi\subset  \bigcup_{\xi\in I^\ast_{\ib,j}} U(\ga(\xi), \f {c_2} n)\subset E_{\ib,j}.
	\end{equation} 
	Thus, setting $e=e_{d+1}$, and using Lemma~\ref{lem-8-1},  we have 
	\begin{align*}
		\sup_{0<s\leq  \f 1n}& \|\tr_{s \vi_{\Og} (e_i,\cdot)e_i }^r (f, G,\cdot)\|_{L^p(G)}^p\leq  C n\int_0^{\f 1n} \int_{G }|\tr_{s\vi_G (e, \xi)e}^r (f, G,\xi)|^p\, d\xi ds\\
		&\leq C\sum_{(\ib,j) \in\Ld_n^{d+1}} \sup_{0<s\leq \f {cj^2} {n^3}}  \int_{I_{\bfi, j}^\ast} |\tr_{se}^r(f, I_{\ib, j}^\ast, \xi)|^p d\xi.
	\end{align*}
	However, by~\eqref{9-6-0},~\eqref{9-10} and Lemma~\ref{lem-9-1:Dec},  this last sum can be estimated above by a constant multiple of 
	\begin{align*}
		& \sum_{(\ib,j) \in\Ld_n^{d+1}}   \int_{I_{\bfi, j}^\ast} \Bl( \f 1{|E_{\ib, j}|}\int_{E_{\ib,j}} |\tr_{(\eta-\xi)/r}^r(f, \Og, \xi)|^q d\eta\Br)^{\f pq}\, d\xi\\
		&\leq C \sum_{(\ib,j) \in\Ld_n^{d+1}}  \int_{I_{\bfi, j}^\ast} \Bl( \f 1{|U(\xi, \f {A_1} n)|}\int_{U(\xi, \f {A_1} n)} |\tr_{(\eta-\xi)/r}^r(f, \Og, \xi)|^q d\eta\Br)^{\f pq}\, d\xi
		\leq C \tau_r (f, \f {A_1} n)_{p,q}^p,
	\end{align*}
	where $A_1:= 2(c_2+c_5+c_6)$. This completes the proof.
\end{proof}

\section{Inverse inequality for $1\leq p\leq \infty$}\label{sec:15}

The main purpose in this section is to show Theorem~\ref{inverse-thm}, the inverse theorem. By Theorem~\ref{thm-9-1-00}, $\og^r_\Og(f, t)_p\leq C_{p,q} \tau_r(f, At)_{p,q}$ for $1\leq q\leq p\leq \infty$,  where $\tau_r(f,t)_{p,q}$ is the $(q,p)$-averaged modulus of smoothness given in~\eqref{eqn:ivanov}. Thus, it is sufficient to prove 
\begin{thm} \label{thm-15-1}If $r\in\NN$, $A>0$, $1\leq q\leq  p\leq \infty$ and $f\in L^p(\Og)$, then 
	$$\tau_r (f, An^{-1})_{p,q} \leq C_{r,A} n^{-r} \sum_{s=0}^n (s+1)^{r-1} E_s (f)_p.$$
	
\end{thm}
Here we recall that $L^p(\Og)$ denotes  the space $L^p(\Og)$ for $p<\infty$  and  the space $C(\Og)$  for  $p=\infty$.

The proof of Theorem~\ref{thm-15-1} relies on two lemmas. To state these   lemmas, we recall that 
for $t>0$, $\xi\in\Og$ and $f\in L^p(\Og)$, 
$$U( \xi, t):= \{\eta\in \Og:\  \  \rho_\Og(\xi,\eta) \leq t\},$$  and 
$$ w_r (f, \xi,  t)_q : =\begin{cases}
	\displaystyle \Bl( \f 1 {|U(\xi,t)|} \int_{U( \xi,t)} |\tr_{(\eta-\xi)/r} ^r (f,\Og,\xi)|^q \, d\eta\Br)^{\f1q},\  \  & \text{if $1\leq q <\infty$};\\
	\sup_{\eta\in U( \xi,t)} |\tr_{(\eta-\xi)/r}^r (f,\Og,\xi)|,\   \ &\text {if $q=\infty$}.\end{cases}$$
\begin{lem} \label{lem-15-1}  Let $G\subset \Og$ be a domain of special type attached to $\Ga$. 
	If  $r\in\NN$, $A>0$, $1\leq q\leq  p\leq \infty$ and $f\in L^p(\Og)$, then 
	\begin{equation}\label{inverse:15-1} \Bl\| w_r (f, \cdot,  An^{-1})_q \Br\|_{L^p(G)}\leq C_{r,p,A} n^{-r} \sum_{s=0}^n (s+1)^{r-1} E_s (f)_{L^p(\Og)}.\end{equation}
\end{lem}
\begin{proof} By monotonicity, it is enough to consider the case $q=p$. 
	It is  easily seen from the definition that 
	\begin{equation}\label{15-1}\|w_r(f,\cdot, t)_p\|_p\leq C_{p,r}\|f\|_p.\end{equation}
	Without loss of generality, we may assume that
	$$ G:=\{(x,y):\  \  x\in (-b, b)^d,\   \   g(x)-1<y\leq g(x)\},$$
	where $b>0$ and $g\in C^2(\RR^d)$.  We may also assume that $n\ge N_0$, where $N_0$ is a sufficiently large positive integer depending only on $\Og$, since otherwise~\eqref{inverse:15-1} follows directly from  the inequality $\|w_r(f,\cdot, t)_p\|_p\leq C E_0 (f)_p$, which can be obtained from~\eqref{15-1}.
	
	For $0\leq k\leq n$, let $P_k\in\Pi_k^{d+1}$ be such that 
	$\|f-P_k\|_{L^p(\Og)} =E_k(f)_{L^p(\Og)}$. Let $m\in\NN$ be such that $2^{m-1} \leq n <2^m$. Then by~\eqref{15-1}, we have 
	\begin{align*}
		\Bl\| w_r (f, \cdot,  An^{-1})_p \Br\|_{L^p(G)}&\leq \Bl\| w_r (f-P_{2^m}, \cdot,  An^{-1})_p\Br\|_{L^p(G)}+\Bl\| w_r (P_{2^m}, \cdot,  An^{-1})_p \Br\|_{L^p(G)}\\
		&\leq  C \|f-P_{2^m}\|_{L^p(\Og)} +  \sum_{j=0}^{m-1} \Bl\| w_r (P_{2^{j+1}}-P_{2^j}, \cdot,  An^{-1})_p\Br\|_{L^p(G)}.
	\end{align*}
	Thus, for the proof of~\eqref{inverse:15-1},  it suffices to show that for each $P\in\Pi_k^{d+1}$,
	\begin{equation}\label{15-2}
		\Bl\| w_r (P, \cdot, An^{-1})_p \Br\|_{L^p(G)}\leq C n^{-r} k^r \|P\|_{L^p(\Og)},
	\end{equation}
	where here and below $C$ and constants in the equivalences may depend on $A$.

	To show~\eqref{15-2}, we first  recall  the following  partition  
	of the domain $\overline{G}$ constructed in Section~\ref{sec:5}: $
	\overline{G}=\bigcup_{\bfi\in\Ld_n^d} \bigcup_{j=0}^{n-1} I_{\mathbf{i},j}$,   where   
	\begin{equation*}\label{15-4-0}
		I_{\mathbf{i},j}:=\Bl\{ (x, y):\  \  x\in \Delta_{\bfi},\  \   g(x)-y\in [\al_{j}, \al_{j+1}]\Br\}
	\end{equation*}
	and 
	\begin{align*}
		\bfi&=(i_1,\dots, i_d)\in \Ld^d_n:=\{ 0, 1,\dots, n-1\}^d\subset \ZZ^d,\\
		\Delta_{\bfi}:&=[t_{i_1}, t_{i_1+1}]\times \dots [t_{i_{d}}, t_{i_{d}+1}]  \   \  \ \text{with}\  \    t_{i}=-b+\f {2i}n b,\\
		\al_j:&= \sin^2 (\f {j\pi}{2\ell_1 n})/(\sin^2\f \pi{2\ell_1}), \   \   j=0,1,\dots, \ell_1 n,\\
		&\text{ with $\ell_1>1$  being  a large  integer parameter}.
	\end{align*}
	As in  Section~\ref{Sec:8}, we also define for any   two given integer parameters $m_0, m_1>1$,  
	\begin{align*}
		\Delta_{\bfi}^\ast&=\Delta_{\bfi,m_0}^\ast:=[t_{i_1-m_0}, t_{i_1+m_0}]\times [t_{i_2-m_0}, t_{i_2+m_0}]\times \dots\times [t_{i_{d}-m_0}, t_{i_{d}+m_0}],\\
		I_{\bfi,j}^\ast:&=I_{\bfi,j, m_0,m_1}^\ast:=\Bl\{ (x, y):\  \ x\in \Delta_{\bfi}^\ast,\   \  \al^\ast_{j-m_1}\leq g(x)-y\leq \al^\ast_{j+m_1}\Br\},
	\end{align*}
	where
	$\al_j^\ast =\al_j$ if $0\leq j\leq n$,  $\al_j^\ast =0$ if $j<0$ and $\al_j^\ast =2$ if $j>n$.  By Proposition~\ref{metric-lem}, we may choose the parameters $m_0, m_1$ large enough so that 
	\begin{equation}\label{15-4}U(\xi, An^{-1}) \subset I_{\bfi, j}^\ast\   \   \text{whenever $\xi\in I_{\bfi,j}$}.\end{equation}
	Note  that for $(\bfi, j)\in\Ld_n^{d+1}$ and  $(x,y) \in I_{\bfi,j}^\ast$, 
	\begin{align}
		\al_j& \sim \f {j^2}{n^2} \sim \da(x,y):=g(x)-y,\    \   j\ge 1, \notag\\
		\al_{j+1}-\al_j &\leq C \f {j+1}{n^2} \leq  \f Cn \vi_n(x,y):=\f C n \bl( \f 1n+\sqrt{\da(x,y)}\br).\label{15-5}
	\end{align}

	Now we turn to the proof of~\eqref{15-2}. 
	Let $P\in \Pi_k^d$ and $1\leq p<\infty$. Then  using Remark~\ref{rem-6-2},  Proposition~\ref{metric-lem} and~\eqref{15-4}, we have 
	\begin{align*}
		&\Bl\| w_r (P, \cdot,  An^{-1})_p \Br\|^p_{L^p(G)}
		\leq C \sum_{(\bfi,j)\in\Ld_n^d} \int_{I_{\bfi,j}}  \f 1 {|I_{\bfi,j}^\ast|} \int_{I_{\bfi,j}^\ast(\xi)} |\tr_{(\eta-\xi)/r} ^r (P,\Og,\xi)|^p d\eta \, d\xi,
	\end{align*}
	where 
	$I_{\bfi,j}^\ast(\xi)=\{\eta\in I_{\bfi,j}^\ast:\  \  [\xi,\eta]\in\Og\}.$
	Note that by H\"older's inequality,
	\begin{align*}
		|\tr_{(\eta-\xi)/r}^r (f,\Og, \xi)|^p & \leq \int_{[0,1]^r} \Bl|\p_{(\eta-\xi)/r}^r f(\xi+ r^{-1}(\eta-\xi) (t_1+\dots+t_r))\Br|^p\, dt_1\dots dr_r\\&\leq C \int_{0}^1 \Bl|\p_{\eta-\xi}^r f(\xi+ t(\eta-\xi) )\Br|^p\, dt.
	\end{align*}
	Thus,
	\begin{align}
		&\Bl\| w_r (P, \cdot,  An^{-1})_p \Br\|^p_{L^p(G)}\notag\\
		&\leq C  \sum_{(\bfi,j)\in\Ld_n^d} \int_{I_{\bfi,j}}  \f 1 {|I_{\bfi,j}^\ast|} \int_{I_{\bfi,j}^\ast(\xi)} \int_0^1 \Bl|\p_{\eta-\xi}^r P(\xi+ t(\eta-\xi) )\Br|^p\, dt
		d\eta \, d\xi.\label{15-6}
	\end{align}
	
	To estimate the sum in this last equation,  we shall   use the Bernstein inequality stated in Theorem~\ref{cor-11-2}.   For convenience, given a parameter $\mu>1$, and two nonnegative integers $l_1, l_2$, we define 
	\begin{align*}
		M_{\mu,n}^{l_1,l_2} f(\xi) := &\max_{ u\in \Xi_{n,\mu}(\xi)}\max_{\zeta\in\sph} \Bigl|  ( z_\zeta(u)\cdot \nabla )^{l_1}\partial_{d+1}^{l_2}f(\xi)\Br|,\  \   \xi\in G,\  \  f\in C^\infty(\Og),	\end{align*}
	where $z_\zeta(u)=(\zeta, \p_\zeta g(u))$, and 
	$$\Xi_{n, \mu} (\xi):= \Bl\{ u\in [-2 a, 2 a]^d:\  \  \|u-\xi_x\|\leq \mu \vi_n(\xi)\Br\}.$$
	We choose the parameter $\mu$ large enough so that 
	$\Delta_{\bfi, 4m_0}^\ast \subset \Xi_{n,\mu}(\xi)$ for any $\xi\in I_{\bfi, j}^\ast$.
	By Theorem~\ref{cor-11-2}, 
	we have
	\begin{equation}\label{15-7}
		\| \vi_n^{l_2} M_{\mu,n}^{l_1,l_2} P\|_{L^p(G^\ast)} \leq C k^{l_1+l_2} \|P\|_{L^p(\Og)},\   \   \ \forall P\in\Pi_k^{d+1}.
	\end{equation}

	Now fix temporarily   $\xi=(\xi_x, \xi_y)\in I_{\bfi,j}$ and $\eta=(\eta_x, \eta_y)\in I_{\bfi, j}^\ast$.  Then $\|\xi_x-\eta_x\|\leq \f cn$, and
	$$ \eta_y-\xi_y =\eta_y-g(\eta_x)+g(\eta_x)-g(\xi_x)+g(\xi_x) -\xi_y.$$
	By  the mean value theorem, there exists $u\in [\xi_x, \eta_x]$ such that 
	$$ \|(\eta_y-\xi_y)-\nabla g(u) \cdot (\eta_x-\xi_x)\|\leq \al_{j+m_1}^\ast -\al_{j-m_1}^\ast \leq c_1 \f {\vi_n(\xi)}n,$$
	where 
	the last step uses~\eqref{15-5}. Thus, setting $\zeta=n(\eta_x-\xi_x)$, 
	we  have $\|\zeta\|\leq c$ and we may write $\eta-\xi$ in the form
	$$ \eta-\xi=\f 1n \Bl(\zeta, \p_{\zeta} g(u)  +s\vi_n(\xi)\Br),$$
	with 
	$$ s=\f {n(\eta_y-\xi_y) -\p_\zeta g(u)}{\vi_n(\xi)}\in [-c_1, c_1].$$
	It follows that  
	$$ \p_{\eta-\xi} =(\eta-\xi)\cdot \nabla = \f 1n \Bl(\p_{z_\zeta (u)} + s \vi_n(\xi) \p_{d+1}\Br),$$
	where $z_\zeta (u) = (\zeta, \p_{\zeta} g(u))$. 
	This  implies that  for any $(x,y)\in I_{\bfi,j}^\ast$, 
	\begin{align*}
		|\p_{\eta-\xi}^r P(x,y)|&\leq C n^{-r} \max _{0\leq k\leq r}\vi_n(\xi)^k |\p_{z_\zeta(u)}^{r-k} \p_{d+1}^{k} P(x,y)| \\
		&\leq C n^{-r} \max _{0\leq k\leq r}\vi_n(x,y)^k   M_{\mu,n}^{r-k,k} P(x,y).\end{align*}
	Thus, setting 
	$$ P_\ast (x,y):=\max _{0\leq k\leq r}\vi_n(x,y)^k   M_{\mu,n}^{r-k,k} P(x,y),$$
	we obtain from~\eqref{15-6} that 
	\begin{align*}
		&\Bl\| w_r (P, \cdot,  An^{-1})_p \Br\|^p_{L^p(G)}
		\leq C n^{-rp}  \sum_{(\bfi,j)\in\Ld_n^d} \int_{I_{\bfi,j}}  \f 1 {|I_{\bfi,j}^\ast|} \int_{I_{\bfi,j}^\ast(\xi)} \int_0^1 |P_\ast(\xi+t(\eta-\xi))|^p\, dt
		d\eta \, d\xi\\
		&\leq C n^{-rp}  \sum_{(\bfi,j)\in\Ld_n^d} \int_{I_{\bfi,j, 2m_0, 2m_1}^\ast}  |P_\ast(\eta)|^p
		d\eta\leq C n^{-rp} \|P_\ast \|_{L^p(G_\ast(2))}^p \leq C \Bl( \f kn\Br)^{rp} \|P\|_{L^p(\Og)}^p,
	\end{align*}
	where the last step uses~\eqref{15-7}. This proves~\eqref{15-2} for $1\leq p<\infty$. 
	
	Finally,~\eqref{15-2} for  $p=\infty$ can be proved  similarly. This completes the proof of Lemma~\ref{lem-15-1}.
\end{proof}

\begin{lem} \label{lem-15-2}  Let $\va\in (0,1)$ and 
	$\Og^\va:=\{\xi\in \Og:\  \  \dist(\xi, \Ga) >\va\}.$
	If  $r\in\NN$, $A>0$, $1\leq q\leq  p\leq \infty$ and $f\in L^p(\Og)$, then 
	\begin{equation*}\label{inverse:15-8} \Bl\| w_r (f, \cdot,  An^{-1})_q \Br\|_{L^p(\Og^\va)}\leq C_{r,p,A} n^{-r} \sum_{s=0}^n (s+1)^{r-1} E_s (f)_{L^p(\Og)}.\end{equation*}
\end{lem}
\begin{proof} 
	The proof is similar to that of Lemma~\ref{lem-15-1}, and in fact, is simpler. It relies on the following Bernstein inequality,
	$$ \|\p^{\b}  P\|_{L^p(\Og^\va)} \leq C k^{|\b|} \|P\|_{L^p(\Og)},\  \  \forall P\in\Pi_k^{d+1},\   \  \forall \b\in\ZZ_+^{d+1},$$
	which is a direct consequence of  the univariate Bernstein  inequality~\eqref{markov-bern}.
\end{proof}

Now we are in a position to prove Theorem~\ref{thm-15-1} .

\begin{proof}[Proof of Theorem~\ref{thm-15-1}]
	By monotonicity, 	it suffices to consider the case $p=q$.
	By Lemma~\ref{lem-2-1-18},  there exist $\va\in (0,1)$ and domains $G_1,\dots, G_{m_0}\subset \Og$ of special type attached to $\Ga$ such that 
	$$\Ga_\va:=\{ \xi\in\Og:\  \  \dist(\xi, \Ga) \leq \va\} \subset \bigcup_{j=1}^{m_0} G_j.$$
	Setting  $\Og^\va:=\Og\setminus \Ga_\va$, we have 
	\begin{align*}
		\tau_r(f, An^{-1})_{p,p} &\leq \sum_{j=1}^{m_0} \|w_r (f, \cdot, An^{-1})_p\|_{L^p(G_i)} + \|w_r (f, \cdot, An^{-1})_p\|_{L^p(\Og^\va)},
	\end{align*}
	which,  using Lemma~\ref {lem-15-1}
	and  Lemma~\ref {lem-15-2}, is estimated above by a constant multiple of 
	\begin{align*}
		n^{-r}\sum_{s=0}^n (s+1)^{r-1} E_s (f)_{L^p(\Og)}.
	\end{align*}
	This completes the proof. 
\end{proof}

\section*{Acknowledgment}
The first named author would like to thank Professor  K. G. Ivanov very much  for kindly  explaining    the works of  
\cite{Iv} to him. The authors are extremely grateful to the anonymous referee for the numerous useful comments.

\end{document}